\documentclass[12pt,a4paper]{article}
\usepackage{latexsym,amssymb,amsmath,mathrsfs,amsthm,mathtools}
\usepackage{sectsty}
\usepackage{upgreek}
\usepackage{color}
\usepackage{marvosym}
\usepackage{wasysym}
\usepackage{bm}
 \usepackage{appendix}
\newtheorem{thm}{Theorem}[section]
\newtheorem{prop}[thm]{Proposition}
\newtheorem{cor}[thm]{Corollary}
\newtheorem{lem}[thm]{Lemma}

\newtheorem{defn}[thm]{Definition}
\newtheorem{remark}[thm]{Remark}
\newtheorem{example}[thm]{Example}

\makeatletter \@addtoreset{equation}{section} \makeatother

\renewcommand{\P}{\mathbb{P}}
\newcommand{\E}{\mathbb{E}}

\newcommand{\R}{\mathbb{R}}

\newcommand{\supp}{\mathop{\rm supp}\nolimits}

\renewcommand{\d}{\mathrm{d}}
\newcommand{\m}{\mathfrak{m} }

\newcommand{\loc}{{\rm loc}}
\newcommand{\1}{{\bf 1}}
\newcommand{\eps}{{\varepsilon}}
\newcommand{\wt}{\widetilde}
\newcommand{\wh}{\widehat}

\renewcommand{\wh}{\widehat}
\renewcommand{\wt}{\widetilde}

\title{\large\bf A remark on subharmonicity for symmetric Dirichlet forms}
\author{
Kazuhiro Kuwae\footnote{\Letter\, Kazuhiro Kuwae ({\tt kuwae@fukuoka-u.ac.jp})
Department of Applied Mathematics, Fukuoka University, Fukuoka 814-0180, Japan. 
Supported in part by JSPS Grant-in-Aid for Scientific Research (S) (No. 22H04942) and fund (No.~215001) from the Central Research Institute of Fukuoka University.}\ \ \ \ 
Rong Lei\footnote{Rong Lei ({\tt leirong17@mails.ucas.edu.cn}) 
Department of Applied Mathematics, Fukuoka University,
Fukuoka 814-0180, Japan. 
Supported in part by JSPS Grant-in-Aid for Scientific Research (S) (No. A22H04942).
Present address: Mathematical Institute, Tohoku University, 
Sendai 980-8577, Japan.}
\ \ and\ \
Ludovico Marini \footnote{Ludovico Marini ({\tt marini@fukuoka-u.ac.jp}) 
Department of Applied Mathematics, Fukuoka University,
Fukuoka 814-0180, Japan. Supported in part by JSPS Grant-in-Aid for Scientific Research (S) (No. 22H04942).}
}
\date{}

\begin{document}
\maketitle

\begin{abstract}
We remove the local boundedness for 
$\mathscr{E}_{\alpha}$-subharmonicity in the framework of  
(not necessarily strongly local) regular symmetric Dirichlet form $(\mathscr{E},D(\mathscr{E}))$ with $\alpha\geq0$ and establish 
the stochastic characterization for $\mathscr{E}$-subharmonic functions without assuming the local boundedness.
\end{abstract}



\section{Introduction}\label{sec:Intro}
In this paper, we give a stochastic characterization of \lq\lq subharmonic functions\rq\rq\;without assuming its local boundedness in the framework of 
general symmetric regular Dirichlet forms over locally compact separable metric spaces. 
In the classical setting, such a characterization between harmonic functions and (local) martingales can be explained by It\^o's formula for Brownian motions on $\R^d$. On $\R^d$, any harmonic function in the ordinary sense is  smooth, hence locally bounded (harmonicity in distributional sense is not necessarily locally bounded, see Example~\ref{ex:BM}). In \cite{Chen}, Z.-Q.~Chen generalized this characterization 
for locally bounded $\mathscr{E}$-harmonic functions in the framework of regular 
Dirichlet form $(\mathscr{E},D(\mathscr{E}))$ associated with general 
symmetric Markov (not necessarily diffusion) processes by using the 
stochastic calculus for additive functionals.   
Later on, this stochastic characterization for locally bounded  $\mathscr{E}$-harmonic functions is extended by Ma--Zhu--Zhu~\cite{MZZ} in the framework of non-symmetric Markov processes. In \cite{CK}, Z.-Q.~Chen and the first named author 
extend \cite{Chen}  
for locally bounded $\mathscr{E}$-subharmonic functions. 
In their results, the local boundedness of $\mathscr{E}$-(sub)harmonicity is assumed a priori, since their assumption \cite[(2.4)]{Chen} or \cite[(2.2)]{CK}, (not \cite[(2.5)]{CK}) for $\mathscr{E}$-subharmonic functions requires 
the local boundedness for the equivalence among \cite[(2.4)]{Chen},  \cite[(2.2)]{CK} and \cite[(2.5)]{CK}.  
This local boundedness is unnecessary at least for diffusion processes, 
because the equivalence among \cite[(2.4)]{Chen}, \cite[(2.2)]{CK} and \cite[(2.5)]{CK} for local Dirichlet forms is automatically satisfied without assuming 
the local boundedness.   
Moreover, by \cite[Lemma~3.6]{CK}, we see that any non-negative $\mathscr{E}$-superharmonic 
function $u$ in the domain of (extended) Dirichlet space is always superharmonic 
in the sense of \cite[Definition~2.1]{CK} 
without assuming its local boundedness. 
This is the reason why we reformulate the results 
for \cite{Chen,CK} without assuming the local boundedness. 

The structure of this paper is as follows: In Section~\ref{sec:Frame}, we state the general setting in terms of Dirichlet forms $(\mathscr{E},D(\mathscr{E}))$ associated to our symmetric Markov processes. 
In Section~\ref{sec3:subharmonicity}, we 
prepare several conditions on functions and lemmas to formulate  
the notion of $\mathscr{E}$-subharmonic functions in terms of $(\mathscr{E},D(\mathscr{E}))$, which is a natural extended notion of 
subharmonicity in distributional sense in the classical case. In this formulation, we do not assume the local boundedness of $\mathscr{E}$-subharmonic functions, but we assume an integrability condition as in \cite[(2.5)]{CK}.    
In Section~\ref{sec:subharmonicity}, we define the two types of subharmonicity in stochastic sense as defined in \cite{CK} and prepare a condition \eqref{eq:2.3} 
instead of \cite[(2.5)]{Chen} or \cite[(2.3)]{CK}. We clarify that 
\cite[(2.5)]{Chen} and \cite[(2.3)]{CK} for locally bounded $u\in D(\mathscr{E})$ 
satisfies \eqref{eq:2.3} (Proposition~\ref{prop:Sufficient}(4)).  
At last, we state the main results (Theorem~\ref{thm:main} and Corollary~\ref{cor:main}). 
In Section~\ref{sec:strongmaxprin}, we state a strong maximum principle in our setting, which is a slight extension of \cite[Theorem~2.11]{CK}. 
In Section~\ref{sec:Example}, we put a few examples. 
    
\section{Framework}\label{sec:Frame}
Let $E$ be a locally compact separable metric space and ${\m}$ a positive Radon measure with full support. 
Let $(\mathscr{E},D(\mathscr{E}))$ 
be a regular Dirichlet form on $L^2(E;{\m})$ 
and ${\bf X}=(\Omega, X_t,\mathscr{F}_t, \mathbb{P}_x)$ the Hunt process 
associated with $(\mathscr{E},D(\mathscr{E}))$ in the sense that 
$P_tf(x)=\mathbb{E}_x[f(X_t)]$ for $f\in L^2(E;{\m})\cap \mathscr{B}(E)$. Here 
$\mathscr{B}(E)$ is the family of Borel measurable functions on $E$ and  
$(P_t)_{t\geq0}$ denotes the $L^2$-semigroup associated with $(\mathscr{E},D(\mathscr{E}))$ (see \cite[\S 7.2]{FOT}). 
An increasing sequence of closed sets 
$\{F_n\}$ is said to be an {\it $\mathscr{E}$-nest} or a 
{\it generalized nest} if $\bigcup_{n=1}^{\infty}D(\mathscr{E}_{F_n})$ is dense in $D(\mathscr{E})$ with respect to the norm $\|\cdot\|_{\mathscr{E}_1^{1/2}}$ defined by $\|u\|_{\mathscr{E}_1^{1/2}}=\mathscr{E}_1(u,u)^{1/2}$ for $u\in D(\mathscr{E})$, and $ D(\mathscr{E}_{F_n}):=\{u\in D(\mathscr{E})\mid u=0\;{\m}\text{-a.e.~on }E\setminus F_n\}$ (see \cite[Chapter III, Definition~2.1(i)]{MR}). 
It is known that an increasing sequence of closed subsets $\{F_n\}$ is an 
$\mathscr{E}$-nest if and only if $\lim_{n\to\infty}{\rm Cap}(K\setminus F_n)=0$ for any compact set $K$ of $E$, where ${\rm Cap}$ is the $1$-order 
capacity with respect to $(\mathscr{E},D(\mathscr{E}))$ (see \cite[Chapter IV, Lemma~4.5(i)$\Leftrightarrow$(iii)]{MR}, 
\cite[\S 2.1 and (2.2.17)]{FOT}).
A subset $N$ of $E$ is said to be {\it $\mathscr{E}$-exceptional} or 
{\it $\mathscr{E}$-polar} if there exists an $\mathscr{E}$-nest $\{F_n\}$ such that $N\subset \bigcap_{n=1}^{\infty}(E\setminus F_n)$. 
A statement $P(x)$ holds {\it $\mathscr{E}$-quasi-everywhere $x\in E$} ({\it $\mathscr{E}$-q.e.~$x\in E$} in short) if the set $\{x\in E\mid P(x)\text{ fails}\}$ is $\mathscr{E}$-exceptional.  
A function $u$ on $E$ is said to be {\it $\mathscr{E}$-quasi-continuous} if there
exists an $\mathscr{E}$-nest $\{F_n\}$ such that $u|_{F_n}$ is continuous on $F_n$ for each $n\in\mathbb{N}$ (see \cite[Chapter III, Definition~2.1(ii)]{MR}). It is known that any $u\in D(\mathscr{E})$ admits an $\mathscr{E}$-quasi-continuous ${\m}$-version $\tilde{u}$ of $u$ (see \cite[Theorem~2.1.3]{FOT} and \cite[Chapter IV, Lemma~4.5(i)$\Leftrightarrow$(iii)]{MR}, or 
\cite[Chapter IV, Proposition~3.3(ii) and \S4 a)]{MR}). 
It is well known that $(\mathscr{E},D(\mathscr{E}))$ can be decomposed into three parts $\mathscr{E}(u,v)=\mathscr{E}^{(c)}(u,v)+\mathscr{E}^{(j)}(u,v)+\mathscr{E}^{(\kappa)}(u,v)$ for 
$u,v\in D(\mathscr{E})\cap C_c(E)$, called \emph{Beurling-Deny decomposition}.
More precisely, there exists a signed finite measure $\mu_{\langle u,v\rangle}^{(c)}$ and positive Radon measures $J$ on $E^2\setminus {\sf diag}$ and $\kappa$ on $E$ such that 
$\mathscr{E}^{(c)}(u,v)=\frac12\mu_{\langle u,v\rangle}^{(c)}(E)$, $\mathscr{E}^{(j)}(u,v)=
\int_{E^2\setminus{\sf diag}}( u (x)- u (y))(v(x)-v(y))J({\d} x{\d} y)$, and 
$\mathscr{E}^{(\kappa)}(u,v)=\int_E u (x) v(x)\kappa({\d} x)$ for $u,v\in D(\mathscr{E})\cap C_c(E)$. Here $C_c(E)$ denotes the family of continuous functions with compact support. 
We set $\mu_{\langle u\rangle}^{(c)}:=\mu_{\langle u,u\rangle}^{(c)}$. 
The measures $\mu_{\langle u,v\rangle}^{(c)}$, $\kappa$ and any marginal measure of $J$ do not charge any $\mathscr{E}$-exceptional set.
Such a decomposition, in fact, extends to general $u,v\in D(\mathscr{E})$ and has a similar expression with $\mathscr{E}^{(j)}(u,v)=
\int_{E^2\setminus{\sf diag}}(\tilde{u}(x)-\tilde{u}(y))(\tilde{v}(x)-\tilde{v}(y))J({\d} x{\d} y)$, and 
$\mathscr{E}^{(\kappa)}(u,v)=\int_E\tilde{u}(x)\tilde{v}(x)\kappa({\d} x)$.
Here and throughout the rest of the paper, $\tilde{u}$ denotes the $\mathscr{E}$-quasi-continuous ${\m}$-version of $u\in D(\mathscr{E})$.   
Let $(N(x,{\d} y), H)$ be a L\'evy system of the Hunt process {\bf X}  (see \cite[Theorem A.3.21]{FOT}).
Let $\mu_H$ be the Revuz measure of the PCAF $H$. 
Then $J$ and $\kappa$ are represented by 
\begin{align*}
J({\d} x{\d} y)=\frac12N(x,{\d} y)\mu_H({\d} x),\quad \kappa({\d} x)=N(x,\partial)\mu_H({\d} x)
\end{align*}
(see \cite[(5.3.6),(5.3.13)]{FOT}).

\section{{$\mathscr{E}$}-subharmonicity}\label{sec3:subharmonicity}
We fix a regular Dirichlet form $(\mathscr{E},D(\mathscr{E}))$ on $L^2(E;{\m})$ and a non-empty 
open subset $D$ of $E$. Let $(\mathscr{E}_D,D(\mathscr{E}_D))$ be 
the part space on $L^2(D;{\m})$ defined by 
$D(\mathscr{E}_D):=\{u\in D(\mathscr{E})\mid \tilde{u}=0\; \mathscr{E}\text{-q.e.~on }E\setminus D\}$ and 
$\mathscr{E}_D(u,v):=\mathscr{E}(u,v)$ for $u,v\in D(\mathscr{E}_D)$. 
Then $(\mathscr{E}_D,D(\mathscr{E}_D))$ becomes a regular Dirichlet form on $L^2(D;\m)$ (see \cite[Theorem~4.4.3]{FOT}).
It is known that $(\mathscr{E}_D,D(\mathscr{E}_D))$ on $L^2(D;{\m})$ is associated with the part process ${\bf X}_D$ in the sense that 
$P_t^Df(x)=\mathbb{E}_x[f(X_t):t<\tau_D]$ for $f\in L^2(D;{\m})\cap \mathscr{B}(D)$. Here ${\bf X}_D:=(\Omega, X_t^D,\mathbb{P}_x)$ is defined to be 
$X_t^D:=X_t$ if $t<\tau_D$ and $X_t^D:=\partial$ if $t\geq\tau_D$ and  
$\mathscr{B}(D)$ is the family of Borel measurable functions on $D$.  
$(P_t^D)_{t\geq0}$ denotes the $L^2$-semigroup 
associated with $(\mathscr{E}_D,D(\mathscr{E}_D))$ and $\tau_D:=\inf\{t>0\mid X_t\notin D\}$ is the first exit time from $D$ with respect to ${\bf X}$. 
We set $D(\mathscr{E}_D)_c:=\{u\in D(\mathscr{E}_D)\mid {\rm supp}[u](\subset D)\text{ is compact}\}$. 
Denote by $(\mathscr{E},D(\mathscr{E})_e)$ the extended Dirichlet space of $(\mathscr{E},D(\mathscr{E}))$ defined by 
\begin{align*}
\left\{\begin{array}{rl}D(\mathscr{E})_e&=\{u\in L^0(E;{\m})\mid \exists \{u_n\}\subset D(\mathscr{E}): \mathscr{E}\text{-Cauchy sequence}\\
&\hspace{4cm}  \text{ such  that }\lim_{n\to\infty}u_n=u\;{\m}\text{-a.e.}\}, \\
\mathscr{E}(u,u)&=\lim_{n\to\infty}\mathscr{E}(u_n,u_n),\end{array}\right.
\end{align*}
(see \cite[p.~40]{FOT} for details on extended Dirichlet space). 
Here $L^0(E;{\m})$ denotes the family of ${\m}$-measurable functions on $E$. 
Any $u\in  D(\mathscr{E})_e$ admits an $\mathscr{E}$-quasi-continuous ${\m}$-version $\tilde{u}$ (see \cite[Theorem~2.1.7]{FOT}).
It is shown in \cite[Theorem~3.4.9]{CFbook} that $D(\mathscr{E}_D)_e=\{f\in D(\mathscr{E})_e\mid \tilde{f}=0\;\mathscr{E}\text{-q.e.~on }E\setminus D\}$.
We define $D(\mathscr{E})_{D,\loc}$ (resp.~$D(\mathscr{E})_{e,D,\loc}$) by
\begin{align*}
D(\mathscr{E})_{D,\loc}:&=\{u\in L^0(E;{\m})\mid \text{ for all }U\Subset D
\text{ there exists }u_U\in D(\mathscr{E})\\
&\hspace{6cm}
\text{ such that }u=u_U\;{\m}\text{-a.e. on }U\},\\
D(\mathscr{E})_{e,D,\loc}:&=\{u\in L^0(E;{\m})\mid \text{ for all }U\Subset D
\text{ there exists }u_U\in D(\mathscr{E})_e\\
&\hspace{6cm}
\text{ such that }u=u_U\;{\m}\text{-a.e. on }U\}.
\end{align*}
Here $U\Subset D$ means that $U$ is a relatively compact open set satisfying $\overline{U}\subset D$. When $D=E$, we write $D(\mathscr{E})_{\loc}$ (resp.~$D(\mathscr{E})_{e,\loc}$) instead of 
$D(\mathscr{E})_{E,\loc}$ (resp.~$D(\mathscr{E})_{e,E,\loc}$). We see that $D(\mathscr{E})_{\loc}\subset D(\mathscr{E})_{D,\loc}$ 
and $D(\mathscr{E})_{e,\loc}\subset D(\mathscr{E})_{e,D,\loc}$ always hold. 
 Any $u\in D(\mathscr{E})_{e,D,\loc}$ admits 
an ${\m}$-a.e.~version $\tilde{u}$, which is $\mathscr{E}$-quasi-continuous on $D$, that is, there exists an $\mathscr{E}$-nest $\{F_n\}$ of closed sets such that $\tilde{u}|_{F_n\cap D}$ is continuous on $F_n\cap D$. Indeed, take an increasing sequence $\{U_k\}_{k=1}^{\infty}$ 
such that $U_k\Subset D$ for all $k\in\mathbb{N}$ 
and $\bigcup_{k=1}^{\infty}U_k=D$. 
Then one can construct a common $\mathscr{E}$-nest $\{F_n\}$ such that  
$\widetilde{u_{U_k}}|_{F_n}$ is continuous on each $F_n$ (see \cite[Chapter III, Proposition~3.3]{MR}). 
We define a function $\tilde{u}$ $\mathscr{E}$-q.e.~on $D$   
by $\tilde{u}:=\widetilde{u_{U_k}}$ on $U_k\cap F_n$ for each $k,n$. 
Then $\tilde{u}|_{F_n\cap D}$ is  continuous on $F_n\cap D$ for each $n$. 
It is easy to see 
$D(\mathscr{E}_D)_{\loc}\subset D(\mathscr{E})_{D,\loc}$ and 
$D(\mathscr{E})_{D,\loc}\cap L_{\loc}^{\infty}(D;\m)\subset 
D(\mathscr{E}_D)_{\loc}$ (see \cite[pp.~1185]{CK}).

We further consider narrow subclasses of 
$D(\mathscr{E})_{e,D,\loc}$
as follows:
\begin{align*}
D(\mathscr{E})_{e,D,\loc}^{\dag}:&=
\left\{u\in D(\mathscr{E})_{e,D,\loc}\,\left|\, 
\int_{U\times D\setminus{\sf diag}}(\tilde{u}(x)-\tilde{u}(y))^2J({\d} x{\d} y)<\infty
\text{ for any }U\Subset D\right.\right\},\\
D(\mathscr{E})_{e,D,\loc}^{\diamond}:&
=\left\{u\in D(\mathscr{E})_{e,D,\loc}\,\left|\,
\int_{U\times (D\setminus V)}|\tilde{u}(x)-\tilde{u}(y)|J({\d} x{\d} y)<\infty\text{ for any }
U\Subset V\Subset D\right.\right\}.
\end{align*}
Then, we can easily see 
$D(\mathscr{E})_{e,D_2,\loc}^{\dag}\subset D(\mathscr{E})_{e,D_1,\loc}^{\dag}$ and $D(\mathscr{E})_{e,D_2,\loc}^{\diamond}\subset D(\mathscr{E})_{e,D_1,\loc}^{\diamond}$ for any non-empty open subsets $D_1\subset D_2$.
When $D=E$, we write $D(\mathscr{E})_{e,\loc}^{\dag}$ 
(resp.~$D(\mathscr{E})_{e,\loc}^{\diamond}$) instead of 
$D(\mathscr{E})_{e,E,\loc}^{\dag}$ (resp.~$D(\mathscr{E})_{e,E,\loc}^{\diamond}$). 
We further set $D(\mathscr{E})_{D,\loc}^{\dag}:=D(\mathscr{E})_{e,D,\loc}^{\dag}\cap D(\mathscr{E})_{D,\loc}$ and 
$D(\mathscr{E})_{D,\loc}^{\diamond}:=D(\mathscr{E})_{e,D,\loc}^{\diamond}\cap D(\mathscr{E})_{D,\loc}$. When $D=E$, we omit the symbol $D$ as mentioned above.  
It is easy to see $\1_D\in D(\mathscr{E})_{D,\loc}^{\dag}\subset 
D(\mathscr{E})_{D,\loc}^{\diamond}$.
It is known that $D(\mathscr{E})\subset D(\mathscr{E})_e\subset D(\mathscr{E})_{e,\loc}^{\dag}$ and 
$D(\mathscr{E})_{\loc}\cap L^{\infty}(E;{\m})\subset D(\mathscr{E})_{\loc}^{\dag}$ (see \cite[Remark~2.4(ii),(iii)]{CK}). 
Similarly, we can get $D(\mathscr{E})\subset D(\mathscr{E})_{D,\loc}^{\dag}$ and 
$D(\mathscr{E})_{D,\loc}\cap L^{\infty}(D;{\m})\subset D(\mathscr{E})_{D,\loc}^{\dag}$. Indeed, $D(\mathscr{E})\subset D(\mathscr{E})_{D,\loc}^{\dag}$ is trivial and  
\begin{align*}
D(\mathscr{E})_{D,\loc}\cap L^{\infty}(D;{\m})&\subset  
D(\mathscr{E}_D)_{\loc}\cap L^{\infty}(D;{\m})\\
&\subset D(\mathscr{E}_D)_{\loc}^{\dag}\;\,\text{(applying $(\mathscr{E}_D,D(\mathscr{E}_D))$ instead of $(\mathscr{E},D(\mathscr{E}))$)}\\
&\subset D(\mathscr{E})_{D,\loc}^{\dag}.
\end{align*}

The conditions defining the 
subclass $D(\mathscr{E})_{\loc}^{\diamond}$ of $D(\mathscr{E})_{\loc}$ 
first appeared in \cite[(2.5)]{CK}. 
For $u\in D(\mathscr{E})_{e,D,\loc}\cap L^{\infty}_{\loc}(D;{\m})$, it is easy to see that 
\begin{align}
\int_{U\times (D\setminus V)}|\tilde{u}(x)-\tilde{u}(y)|J({\d} x{\d} y)<\infty\label{eq:finite}
\end{align}
if and only if 
\begin{align}
\int_{U\times (D\setminus V)}
|\tilde{u}(y)|J({\d} x{\d} y)<\infty\label{eq:finite*}
\end{align}
 for any 
$U\Subset V\Subset D$, because $J(U\times V^c)<\infty$ for such $U,V$ (see \cite[Remark~2.4(iii)]{CK}). It is well-known that $J(U\times V^c)<\infty$ for $U\Subset V\Subset E$ (see \cite[Corollary~5.1]{Kw:func},\cite[(2.3)]{Chen}). 
By definition, 
$D(\mathscr{E})_{D,\loc}^{\dag}= D(\mathscr{E})_{D,\loc}^{\diamond}= D(\mathscr{E})_{D,\loc}$ under $J\equiv0$. 

A  function $T:\R^N\to\R$, is said to be a \emph{{\rm(}generalized{\,\rm)} normal contraction} if 
\begin{align*}
T(0)=0\quad\text{ and }\quad |T(x)-T(y)|\leq \sum_{k=1}^N|x_k-y_k|
\end{align*}
for any $x=(x_1,\cdots,x_N)$, $y=(y_1,\cdots, y_N)\in\R^N$. 
When $N=1$, we call such $T:\R\to\R$ a \emph{normal contraction}. 
Note that $\phi(t):=t^{\sharp}:=0\lor t\land 1$, $\phi(t):=|t|$, $\phi(t):=t^+:=t\lor 0$ and $\phi(t):=t^-:=(-t)\lor0$ are typical examples of normal contractions. More generally, any real-valued function 
$\phi_{\eps}$ on $\R$ depending on $\eps>0$ such that $\phi_{\eps}(t)=t$ for $t\in[0,1]$, $-\eps\leq\phi_{\eps}(t)\leq 1+\eps$ for all $t\in\R$, and $0\leq\phi_{\eps}(t)-\phi_{\eps}(s)\leq t-s$ whenever $s<t$, is also a normal contraction. The following lemmas are easy to prove:

\begin{lem}\label{lem:contraction}
Let $T:\R^N\to\R$ be a generalized normal contraction. 
For $u_i\in D(\mathscr{E})_{D,\loc}^{\dag}$ {\rm(}resp.~$u_i\in D(\mathscr{E})_{D,\loc}^{\diamond}$, $u_i\in D(\mathscr{E})_{D,\loc}${\rm)} {\rm(}$i=1,2,\cdots, N${\rm)}, then $T(u_1,\cdots,u_N)\in 
 D(\mathscr{E})_{D,\loc}^{\dag}$ {\rm(}resp.~$T(u_1,\cdots,u_N)\in D(\mathscr{E})_{D,\loc}^{\diamond}$, $T(u_1,\cdots,u_N)\in D(\mathscr{E})_{D,\loc})$. In particular, for $u_1,u_2\in 
 D(\mathscr{E})_{D,\loc}^{\dag}$ {\rm(}resp.~$u_1,u_2\in D(\mathscr{E})_{D,\loc}^{\diamond}$, $u_1,u_2\in D(\mathscr{E})_{D,\loc}${\rm )}, then $u_1\lor u_2, u_1\land u_2\in D(\mathscr{E})_{D,\loc}^{\dag}$ {\rm(}resp.~$u_1\lor u_2, u_1\land u_2\in D(\mathscr{E})_{D,\loc}^{\diamond}$, $u_1\lor u_2, u_1\land u_2\in D(\mathscr{E})_{D,\loc}${\rm)}. 
Similar assertions also hold in terms of $D(\mathscr{E})_{e,\loc}$, $D(\mathscr{E})_{e,\loc}^{\diamond}$, $D(\mathscr{E})_{e,\loc}^{\dag}$. 
\end{lem}

\begin{lem}\label{lem:weldefinedness}
We always have $D(\mathscr{E})_{e,D,\loc}^{\dag}\subset D(\mathscr{E})_{e,D,\loc}^{\diamond}\subset D(\mathscr{E})_{e,D,\loc}$. 
For $u\in D(\mathscr{E})_{e,D,\loc}^{\diamond}$ 
{\rm(}resp.~$u\in D(\mathscr{E})_{e,D,\loc}^{\dag}${\rm)} and 
$v\in D(\mathscr{E}_D)_c\cap L^{\infty}(D;{\m})$ {\rm(}resp.~$v\in D(\mathscr{E}_D)_c${\rm)}, 
the integration 
\begin{align}
\int_{D^2\setminus {\sf diag}}
(\tilde{u}(x)-\tilde{u}(y))(\tilde{v}(x)-\tilde{v}(y))J({\d} x{\d} y)\label{eq:welldefined3}
\end{align}
is well-defined.
Moreover, 
for $u\in 
D(\mathscr{E})_{e,\loc}^{\diamond}
$ {\rm(}resp.~$u\in 
D(\mathscr{E})_{e,\loc}^{\dag}
${\rm)}
and 
$v\in D(\mathscr{E}_D)_c\cap L^{\infty}(D;{\m})$ {\rm(}resp.~$v\in D(\mathscr{E}_D)_c${\rm)}, the integration 
\begin{align}
\int_{E^2\setminus {\sf diag}}
(\tilde{u}(x)-\tilde{u}(y))(\tilde{v}(x)-\tilde{v}(y))J({\d} x{\d} y)\label{eq:welldefined4}
\end{align}
is well-defined.
\end{lem}
\begin{remark}
{\rm The formulation in \cite{Chen,CK} for the 
well-definedness of the integration \eqref{eq:welldefined4} is 
not completely precise clarification, because for any $u\in D(\mathscr{E}_D)_{\loc}$ or $u\in D(\mathscr{E})_{e,D,\loc}$, the value $\tilde{u}$ is not defined on $E\setminus D$, in particular, we cannot discuss the integration 
$\int_{U\times (E\setminus V)}\tilde{u}(y)J({\d} x{\d} y)$ for such $u$.   
To ensure $\tilde{u}$ is well-defined on $E\setminus D$, we always 
assume $u\in D(\mathscr{E})_{e,\loc}^{\diamond}$ for the 
well-definedness of the integration \eqref{eq:welldefined4} above.
}
\end{remark}
\begin{proof}[{\bf Proof of Lemma~\ref{lem:weldefinedness}.}]
For $u\in D(\mathscr{E})_{e,D,\loc}^{\dag}$ and $U\Subset V\Subset D$,
\begin{align*}
\int_{U\times (D\setminus V)}|\tilde{u}(x)-\tilde{u}(y)|J({\d} x{\d} y)&
\leq \left(\int_{U\times (D\setminus V)}|\tilde{u}(x)-\tilde{u}(y)|^2J({\d} x{\d} y) \right)^{\frac12}J(U\times V^c)^{\frac12}\\
&\leq \left(\int_{U\times D\setminus{\sf diag}}|\tilde{u}(x)-\tilde{u}(y)|^2J({\d} x{\d} y) \right)^{\frac12}J(U\times V^c)^{\frac12}<\infty
\end{align*}
implies $u\in D(\mathscr{E})_{e,D,\loc}^{\diamond}$.  
For $u\in D(\mathscr{E})_{e,D,\loc}^{\diamond}$ and $v\in D(\mathscr{E}_D)_c\cap L^{\infty}(D;{\m})$, taking open sets
$U,V$ with  ${\rm supp}[v]\subset U\Subset V\Subset D$ and some 
$u_V\in D(\mathscr{E})_e$ with $u=u_V$ ${\m}$-a.e.~on $V$, we then see 
\begin{align*}
\int_{D^2\setminus{\sf diag}}|\tilde{u}(x)-\tilde{u}(y)|&|\tilde{v}(x)-\tilde{v}(y)|J({\d} x{\d} y)\\
&
=\int_{V^2\setminus{\sf diag}}|\tilde{u}(x)-\tilde{u}(y)||\tilde{v}(x)-\tilde{v}(y)|J({\d} x{\d} y)\\
&\hspace{1cm}+2\int_{V\times (D\setminus V)}|\tilde{u}(x)-\tilde{u}(y)||\tilde{v}(x)-\tilde{v}(y)|J({\d} x{\d} y)\\
& = \int_{V^2\setminus {\sf diag}}|\tilde{u}_V(x)-\tilde{u}_V(y)||\tilde{v}(x)-\tilde{v}(y)|J({\d} x{\d} y)\\
&\hspace{1cm}+2\int_{V\times (D\setminus V)}|\tilde{u}(x)-\tilde{u}(y)||\tilde{v}(x)
|J({\d} x{\d} y)\\
&=\int_{V^2\setminus {\sf diag}}|\tilde{u}_V(x)-\tilde{u}_V(y)||\tilde{v}(x)-\tilde{v}(y)|J({\d} x{\d} y)\\
&\hspace{1cm}+2\int_{U\times (D\setminus V)}|\tilde{u}(x)-\tilde{u}(y)||\tilde{v}(x)
|J({\d} x{\d} y)\\
&\leq \mathscr{E}(u_V,u_V)^{\frac12}\mathscr{E}(v,v)^{\frac12}\\
&\hspace{1cm}+2\|\tilde{v}\|_{L^{\infty}(D;{\m})}
\int_{U\times (D\setminus V)}|\tilde{u}(x)-\tilde{u}(y)|J({\d} x{\d} y)<\infty.
\end{align*}
For $u\in D(\mathscr{E})_{e,D,\loc}^{\dag}$ and $v\in D(\mathscr{E}_D)_c$, the left hand side is estimated above by
\begin{align*}
\mathscr{E}(u_V,u_V)^{\frac12}\mathscr{E}(v,v)^{\frac12}+2\mathscr{E}(v,v)^{\frac12}\left(\int_{U\times D\setminus{\sf diag}}|\tilde{u}(x)-\tilde{u}(y)|^2J({\d} x{\d} y) \right)^{\frac12}<\infty.
\end{align*}
Next we assume $u\in  
D(\mathscr{E})_{e,\loc}^{\diamond}$
and 
$v\in D(\mathscr{E}_D)_c\cap L^{\infty}(D;{\m})$. Take open sets $U,V$ and $u_V\in D(\mathscr{E})_e$ as above.  
Then 
\begin{align*}
\int_{E^2\setminus{\sf diag}}&|\tilde{u}(x)-\tilde{u}(y)||\tilde{v}(x)-\tilde{v}(y)|J({\d} x{\d} y)\\
&
=\int_{V^2\setminus{\sf diag}}|\tilde{u}(x)-\tilde{u}(y)||\tilde{v}(x)-\tilde{v}(y)|J({\d} x{\d} y)\\
&\hspace{1cm}+2\int_{V\times (E\setminus V)}|\tilde{u}(x)-\tilde{u}(y)||\tilde{v}(x)-\tilde{v}(y)|J({\d} x{\d} y)\\
&\leq\int_{V^2\setminus {\sf diag}}|\tilde{u}_V(x)-\tilde{u}_V(y)||\tilde{v}(x)-\tilde{v}(y)|J({\d} x{\d} y)\\
&\hspace{1cm}+2\int_{V\times (E\setminus V)}|\tilde{u}(x)-\tilde{u}(y)||\tilde{v}(x)
|J({\d} x{\d} y)
\end{align*}
\begin{align*}
\qquad\qquad\qquad&\qquad\\
&=\int_{V^2\setminus {\sf diag}}|\tilde{u}_V(x)-\tilde{u}_V(y)||\tilde{v}(x)-\tilde{v}(y)|J({\d} x{\d} y)\\
&\hspace{1cm}+2\int_{U\times (E\setminus V)}|\tilde{u}(x)-\tilde{u}(y)||\tilde{v}(x)
|J({\d} x{\d} y)\\
&\leq \mathscr{E}(u_V,u_V)^{\frac12}\mathscr{E}(v,v)^{\frac12}\\
&\qquad+2\|\tilde{v}\|_{L^{\infty}(D;{\m})}
\int_{U\times (E\setminus V)}|\tilde{u}(x)-\tilde{u}(y)|J({\d} x{\d} y)
<\infty.
\end{align*}
For $u\in D(\mathscr{E})_{e,\loc}^{\dag}$ and $v\in D(\mathscr{E}_D)_c$, the left hand side is estimated above by
\begin{align*}
\mathscr{E}(u_V,u_V)^{\frac12}\mathscr{E}(v,v)^{\frac12}+2\mathscr{E}(v,v)^{\frac12}\left(\int_{U\times E\setminus{\sf diag}}|\tilde{u}(x)-\tilde{u}(y)|^2J({\d} x{\d} y) \right)^{\frac12}<\infty.
\end{align*}
\end{proof}
Recall that $\mu_{\langle u,v\rangle}^{(c)}$ for $u,v\in D(\mathscr{E})$ is the finite signed measure of the bilinear form 
$\mathscr{E}^{(c)}$ as stated in Section~\ref{sec:Frame}. 
Note that $\mu_{\langle u,v\rangle}^{(c)}$ satisfies 
the strong local property in the sense that 
$\mu_{\langle u,v\rangle}^{(c)}=0$ if $u\equiv{\rm const}$ on a 
neighborhood of ${\rm supp}[v]$ (see \cite[Lemma~5.2 ($\Gamma 6$)]{Kw:func}). Moreover, $\mu_{\langle u,v\rangle}^{(c)}$ 
satisfies the Cauchy--Schwarz inequality 
$|\mu_{\langle u,v\rangle}^{(c)}(B)|\leq \mu_{\langle u\rangle}^{(c)}(B)^{1/2}\mu_{\langle v\rangle}^{(c)}(B)^{1/2}$ for any Borel subset $B$ of $E$, because $t^2\mu_{\langle u\rangle}^{(c)}(B)+2t\mu_{\langle u,v\rangle}^{(c)}(B)+\mu_{\langle v\rangle}^{(c)}(B)
=\mu_{\langle tu+v\rangle}^{(c)}(B)\geq0$ for any $t\in\R$. 
\begin{cor}\label{cor:welldefinedenergy}
The quantity $\mathscr{E}_{\alpha}(u,v)$ for $u\in 
D(\mathscr{E})_{\loc}^{\diamond}$ {\rm(}resp.~$u\in 
D(\mathscr{E})_{\loc}^{\dag}${\rm )}
and 
$v\in D(\mathscr{E}_D)_c\cap L^{\infty}(D;{\m})$ 
{\rm(}resp.~$v\in D(\mathscr{E}_D)_c${\rm)}
defined by 
\begin{align*}
\mathscr{E}(u,v):&=\frac12\mu^{(c)}_{\langle u,v\rangle}(D)+\int_{E^2\setminus {\sf diag}}
\hspace{-0.2cm}
(\tilde{u}(x)-\tilde{u}(y))(\tilde{v}(x)-\tilde{v}(y))J({\d} x{\d} y)+\int_D\tilde{u}(x)\tilde{v}(x)\kappa({\d} x),\\
\mathscr{E}_{\alpha}(u,v):&=\mathscr{E}(u,v)+\alpha\int_Duv\,{\d}{\m}
\end{align*}
is well-defined for $\alpha>0$.  
Moreover, the quantity $\mathscr{E}(u,v)$ for $u\in
D(\mathscr{E})_{e,\loc}^{\diamond}$ 
{\rm(}resp.~$u\in
D(\mathscr{E})_{e,\loc}^{\dag}${\rm)}
and 
$v\in D(\mathscr{E}_D)_c\cap L^{\infty}(D;{\m})$ 
{\rm(}resp.~$v\in D(\mathscr{E}_D)_c${\rm)} is also well-defined.
So $\mathscr{E}_{\alpha}$ is a bilinear form on 
$D(\mathscr{E})_{\loc}^{\diamond}\times (D(\mathscr{E}_D)_c\cap L^{\infty}(D;{\m}))$ {\rm(}resp.~$D(\mathscr{E})_{e,\loc}^{\diamond}\times (D(\mathscr{E}_D)_c\cap L^{\infty}(D;{\m}))${\rm)} under $\alpha>0$ {\rm(}resp.~$\alpha=0${\rm)}, and 
 also a bilinear form on $D(\mathscr{E})_{\loc}^{\dag}\times D(\mathscr{E}_D)_c$ {\rm(}resp.~$D(\mathscr{E})_{e,\loc}^{\dag}\times D(\mathscr{E}_D)_c${\rm)} under $\alpha>0$ {\rm(}resp.~$\alpha=0${\rm)}.
When $\alpha=0$, we write $\mathscr{E}(u,v)$ instead of $\mathscr{E}_0(u,v)$.
\end{cor}
\begin{proof}[{\bf Proof.}]
For $u\in D(\mathscr{E})_{e,\loc}^{\diamond}$ and $v\in D(\mathscr{E}_D)_c\cap L^{\infty}(D;{\m})$, taking open sets
$U,V$ with  ${\rm supp}[v]\subset U\Subset V\Subset D$, we then see 
\begin{align*}
|\mu^{(c)}_{\langle u,v\rangle}(D)|&=
|\mu^{(c)}_{\langle u,v\rangle}(U)|=|\mu^{(c)}_{\langle u_U,v\rangle}(U)|\leq 
\mu^{(c)}_{\langle u_U\rangle}(E)^{\frac12}\mu^{(c)}_{\langle v\rangle}(E)^{\frac12}<\infty,\\
 \int_D|\tilde{u}||\tilde{v}|{\d}\kappa&\leq \|\tilde{u}_U\|_{L^2(E;\kappa)}\|\tilde{v}\|_{L^2(E;\kappa)}<\infty,
\end{align*}
where $u_U$ is a function in $D(\mathscr{E})$ (resp.~$D(\mathscr{E})_e$) such that $u=u_U$ ${\m}$-a.e.~on $U$ under $\alpha>0$ (resp.~$\alpha=0$). 
The first equality above follows from the strong local property of $\mu_{\langle u,v\rangle}^{(c)}$. Indeed, let $\{G_n\}$ be an increasing sequence of relatively compact open subsets satisfying $G_n\Subset D$ and $\bigcup_{n=1}^{\infty}G_n=D$. Since $v=0$ on $D\setminus U$, we have  
$\mu^{(c)}_{\langle v\rangle}(G_n\setminus U)=0$ by \cite[Corollary~3.2.1]{FOT}, hence $\mu^{(c)}_{\langle v\rangle}(D\setminus U)=0$. This together with the Cauchy--Schwarz inequality implies $\mu^{(c)}_{\langle u,v\rangle}(D\setminus U)=0$. 
 Combining this with Lemma~\ref{lem:weldefinedness} under 
 $u\in D(\mathscr{E})^{\diamond}_{\loc}$ (resp.~$u\in D(\mathscr{E})_{e,\loc}^{\diamond}$) with $\alpha>0$ (resp.~$\alpha=0$),  
we obtain the conclusion. 
\end{proof}

We give a new criterion for a function belonging to  $D(\mathscr{E})_{\loc}^{\diamond}$: 
\begin{prop}\label{prop:newcriterion}
Let $L(x,{\d} y)$ be a kernel on $E\times\mathscr{B}(E)$. Suppose that the jumping measure $J$ satisfies $J({\d} x{\d} y)\leq D_2L(x,{\d} y)\m({\d} x)$ for some $D_2>0$.
Let ${\sf d}$ be a distance function on $E$ consistent with the given topology. Assume that there exists $\beta\in]0,+\infty[$ such that 
\begin{align*}
M_{\beta}:={\m}\text{\rm-ess-sup}_{x\in E}\int_E(1\land {\sf d}(x,y)^{\beta})L(x,{\d} y)<\infty.
\end{align*} 
Then, for $p\in[1,+\infty]$, {\rm(}resp.~$p\in[2,+\infty]${\rm)}, 
we have 
 $D(\mathscr{E})_{\loc}\cap L^p(E;{\m})\subset D(\mathscr{E})_{\loc}^{\diamond}$ 
and $D(\mathscr{E})_{e,\loc}\cap L^p(E;{\m})\subset D(\mathscr{E})_{e,\loc}^{\diamond}$ {\rm(}resp.~$D(\mathscr{E})_{\loc}\cap L^p(E;{\m})\subset D(\mathscr{E})_{\loc}^{\dag}$ 
and $D(\mathscr{E})_{e,\loc}\cap L^p(E;{\m})\subset D(\mathscr{E})_{e,\loc}^{\dag}${\rm)}.
\end{prop}  
\begin{proof}[{\bf Proof.}]
The assertion for $p=\infty$ is known. Let $p\in[1,+\infty[$. Take $u\in D(\mathscr{E})_{\loc}\cap L^p(E;{\m})$ or 
$u\in D(\mathscr{E})_{e,\loc}\cap L^p(E;{\m})$.  
Let $U\Subset V\subset E$ and set ${\sf d}(\overline{U},V^c):=\inf_{x\in \overline{U},y\in V^c}{\sf d}(x,y)$. 
Since $\overline{U}$ is compact, ${\sf d}(\overline{U},V^c)>0$. 
Take $r\in]0,{\sf d}(\overline{U},V^c)\land1]$. Then
\begin{align*}
\int_{U\times V^c}&|\tilde{u}(x)-\tilde{u}(y)|J({\d} x{\d} y)\\&\leq 
\left(\int_{U\times V^c}|\tilde{u}(x)-\tilde{u}(y)|^pJ({\d} x{\d} y) \right)^{\frac{1}{p}}J(U\times V^c)^{1-\frac{1}{p}}\\
&\leq \left(2^{p-1}\int_{U\times V^c}(|\tilde{u}(x)|^p+|\tilde{u}(y)|^p)J({\d} x{\d} y) \right)^{\frac{1}{p}}J(U\times V^c)^{1-\frac{1}{p}}\\
&=
\left(2^{p-1}\int_U|\tilde{u}(x)|^pJ({\d} x\times V^c)+2^{p-1}\int_{V^c}|\tilde{u}(x)|^pJ({\d} x\times U)
 \right)^{\frac{1}{p}}J(U\times V^c)^{1-\frac{1}{p}}\\
&\leq D_2^{\frac{1}{p}}\left(2^{p-1}\int_U|u(x)|^pL(x,V^c){\m}({\d} x)+2^{p-1}\int_{V^c}|u(x)|^p L(x,U){\m}({\d} x)
 \right)^{\frac{1}{p}}J(U\times V^c)^{1-\frac{1}{p}}\\
 &\leq 2 D_2^{\frac{1}{p}} \left(\frac{M_{\beta}}{r^{\beta}}\right)^{\frac{1}{p}}\|u\|^p_{L^p(E;{\m})}J(U\times V^c)^{1-\frac{1}{p}}<\infty,
\end{align*}
where we use 
\begin{align*}
L(x,V^c)&\leq L(x,B_r(x)^c)=\int_{{\sf d}(x,y)\geq r}\left(1\land \frac{{\sf  d}(x,y)^{\beta}}{r^{\beta}} \right)L(x,{\d} y)\\
&\leq \int_E\left(1\land \frac{{\sf  d}(x,y)^{\beta}}{r^{\beta}} \right)L(x,{\d} y)\leq\frac{1}{r^{\beta}}M_{\beta}\quad{\m}\text{-a.e.~}x\in U
\end{align*}
and $L(x,U)\leq \frac{1}{r^{\beta}}M_{\beta}$ ${\m}$-a.e.~on $V^c$. 
When $p\in[2,+\infty[$, we can similarly obtain
\begin{align*}
\int_{U\times V^c}|\tilde{u}(x)-\tilde{u}(y)|^2J({\d} x{\d} y)<\infty.
\end{align*}
\end{proof}
Next corollary corresponds to \cite[Lemma~2.2]{MasamuneUemura2}.
\begin{cor}\label{cor:newcriterion}
Let $L(x,{\d} y)$ be a kernel on $E\times\mathscr{B}(E)$.
Suppose that the jumping measure $J$ satisfies $J({\d} x{\d} y)\leq D_2L(x,{\d} y){\m}({\d} x)$ for some $D_2>0$. 
Let ${\sf d}$ be a distance function consistent with the given topology. Assume 
that there exists $\beta\in]0,+\infty[$ such that 
\begin{align}
M_{\beta}:={\m}\text{\rm-ess-sup}_{x\in E}\int_E(1\land {\sf d}(x,y)^{\beta})L(x,{\d} y)<\infty.\label{eq:MbetaFinite}
\end{align} 
Then, for $u\in D(\mathscr{E})_{\loc}\cap L^p(E;{\m})$ with $p\in[1,+\infty]$ and $v\in D(\mathscr{E})_c\cap L^{\infty}(E;{\m})$, 
the quantity $\mathscr{E}_{\alpha}(u,v)$ for $\alpha\geq0$ is well-defined. Moreover, 
for $u\in D(\mathscr{E})_{e,\loc}\cap L^p(E;{\m})$ with $p\in[1,+\infty]$ and $v\in D(\mathscr{E})_c\cap L^{\infty}(E;{\m})$, 
the quantity $\mathscr{E}(u,v)$ is well-defined. 
\end{cor} 
\begin{proof}[{\bf Proof.}]
The assertion is a combination of Corollary~\ref{cor:welldefinedenergy}
and Proposition~\ref{prop:newcriterion}.
\end{proof}
\begin{prop}
\label{cor:Criteion}
Let $L(x,{\d} y)$ be a kernel on $E\times\mathscr{B}(E)$ and ${\sf d}$ be a distance function consistent with the underlying topology.
Assume there exists $D_2, C > 0$ such that:
\begin{enumerate}
    \item $L(x,{\d} y) 
    \leq D_2 \frac{{\m}({\d} y)}{{\sf d}(x,y)^p}$,
    \item ${\m}(B_r(x))\leq Cr^d$ for every $x \in E$ and $r > 0$,
\end{enumerate}
for some $d, p \in \R$, $p >d\geq0$. Here $B_r(x):=\{y\in E\mid {\sf d}(x,y)<r\}$. Then $M_\beta < \infty$, i.e.,
\eqref{eq:MbetaFinite} holds, for any $\beta > p - d$.
\end{prop} 
\begin{proof}[{\bf Proof.}]
We calculate 
\begin{align*}
\int_E&(1\land {\sf d}(x,y)^{\beta})L(x,{\d} y)\\
&=D_2
\int_E\frac{(1\land {\sf d}(x,y)^{\beta})}{{\sf d}(x,y)^p}{\m}({\d} y)\\
&=D_2\sum_{n=0}^{\infty}\int_{1/2^{n+1}\leq{\sf d}(x,y)<1/2^n}{\sf d}(x,y)^{\beta-p}{\m}({\d} x)+D_2\sum_{n=0}^{\infty}\int_{2^n\leq{\sf d}(x,y)<2^{n+1}}{\sf d}(x,y)^{-p}{\m}({\d} x)\\
&\leq D_2\left\{\begin{array}{ll}\displaystyle{{\m}(B_1(x))}, & 
\displaystyle{(p\leq\beta)} \\ \displaystyle{\sum_{n=0}^{\infty}(2^{n+1})^{p-\beta}{\m}(B_{\frac{1}{2^n}}(x))}, & 
\displaystyle{(\beta<p<d+\beta)}\end{array}\right. 
+ D_2\sum_{n=0}^{\infty}(2^n)^{-p}{\m}(B_{2^{n+1}}(x))\\
&\leq D_2
\left\{\begin{array}{ll}\displaystyle{C},  & \displaystyle{(p\leq\beta)}
 \\
 \displaystyle{C2^{p-\beta}\sum_{n=0}^{\infty}\left(\frac{1}{2^n}\right)^{d-p+\beta}},
  & \displaystyle{(\beta<p<d+\beta)}\end{array}\right. 
+ C2^dD_2\sum_{n=0}^{\infty}\left(\frac{1}{2^n}\right)^{p-d}<\infty.
\end{align*}
Note here that the last quantity is independent of $x\in E$. 
\end{proof}

Now we define the notion of $\mathscr{E}_{\alpha}$-subharmonicity in $D$ as follows: 

\begin{defn}[{{\bf {\boldmath$\mathscr{E}_{\alpha}$}-subharmonicity}}]\label{df:subharmonicity}
{\rm Fix $\alpha>0$ and an open set $D$. A function $u:E\to\R$ is called \emph{$\mathscr{E}_{\alpha}$-subharmonic} (resp.~\emph{$\mathscr{E}_{\alpha}$-superharmonic}) \emph{in $D$} if $u\in 
D(\mathscr{E})_{\loc}^{\diamond}$ 
and 
\begin{align}
\mathscr{E}_{\alpha}(u,v)\leq0\quad \text{\rm (resp.~}\geq0\text{\rm )}\label{eq:Subharmonicity}
\end{align}  
for all $v\in D(\mathscr{E}_D)_c\cap L^{\infty}(D;{\m})_+$. 
Here $L^{\infty}(D;{\m})_+$ denotes the family of non-negative bounded ${\m}$-measurable function on $D$. 
A function $u:E\to\R$ is called \emph{$\mathscr{E}_{\alpha}$-harmonic in $D$} if it is 
 $\mathscr{E}_{\alpha}$-subharmonic and $\mathscr{E}_{\alpha}$-superharmonic in $D$. 
A function $u:E\to\R$ is called \emph{$\mathscr{E}$-subharmonic} (resp.~\emph{$\mathscr{E}$-superharmonic}) \emph{in $D$} if $u\in 
D(\mathscr{E})_{e,\loc}^{\diamond}$ 
and 
\begin{align}
\mathscr{E}(u,v)\leq0\quad \text{\rm (resp.~}\geq0\text{\rm )}\label{eq:Subharmonicity*}
\end{align}  
for all $v\in D(\mathscr{E}_D)_c\cap L^{\infty}(D;{\m})_+$. 
A function $u:E\to\R$ is called \emph{$\mathscr{E}$-harmonic in $D$} if it is 
 $\mathscr{E}$-subharmonic and $\mathscr{E}$-superharmonic in $D$. 
}
\end{defn} 
It is easy to see that $\1\in 
D(\mathscr{E})_{\loc}^{\dag}\subset 
D(\mathscr{E})_{\loc}^{\diamond}$ is $\mathscr{E}_{\alpha}$-superharmonic in $D$ under $\alpha\geq0$. Moreover, it is $\mathscr{E}$-harmonic in $D$ provided $\kappa(D)=0$.

Let $\mathscr{C}$ denote the special standard core of $(\mathscr{E},D(\mathscr{E}))$. More precisely, $\mathscr{C}$ is a dense linear subalgebra of $C_c(E)$ such that for any compact set $K$ and a relatively compact open set $G$ with $K\subset G$ there exists non-negative 
$u\in\mathscr{C}$ such that $u=1$ on $K$ and $u=0$ on $E\setminus G$,   
and for each $\varepsilon>0$ there exists a real valued function $\varphi_{\varepsilon}$ satisfying $\varphi_{\varepsilon}(t)=t$ for $t\in[0,1]$, $-\varepsilon\leq\varphi_{\varepsilon}\leq 1+\varepsilon$ on $\R$, and $0\leq \varphi_{\varepsilon}(t)-\varphi_{\varepsilon}(s)\leq t-s$ 
whenever $s<t$ (see \cite[p.~6]{FOT}). For any open subset $G$ of $E$, we set 
$\mathscr{C}_G:=\{u\in\mathscr{C}\mid u=0\text{ on }E\setminus G\}$ and 
$(\mathscr{C}_G)_+:=\{u\in\mathscr{C}_G\mid u\geq0\text{ on }E\}$.
Next theorem generalizes \cite[Theorem~2.8]{CK} without assuming the local boundedness on $D$ for $u\in 
D(\mathscr{E})_{\loc}^{\diamond}$. 

\begin{thm}\label{thm:CharacSubhamonicity}
Fix $\alpha\geq0$. For $u\in 
D(\mathscr{E})_{\loc}^{\diamond}$,
the following are equivalent:
\begin{enumerate}
\item[\rm(1)] $u$ is $\mathscr{E}_{\alpha}$-subharmonic in $D$. 
\item[\rm(2)] $u$ satisfies \eqref{eq:Subharmonicity} for any $v\in D(\mathscr{E}_D)\cap C_c(D)_+$.
\item[\rm(3)] $u$ satisfies \eqref{eq:Subharmonicity} for any $v\in(\mathscr{C}_{D})_+$.
\end{enumerate} 
Moreover, for $u\in 
D(\mathscr{E})_{e,\loc}^{\diamond}$, a similar equivalence also 
holds under $\alpha=0$. 
\end{thm}
\begin{proof}[{\bf Proof.}]
The implications (1)$\Longrightarrow$(2)$\Longrightarrow$(3) are trivial, thus we only need to prove (3)$\Longrightarrow$(1). 
Take $v\in D(\mathscr{E}_D)_c\cap L^{\infty}(D;{\m})_+$.
For the case $v\equiv0$, the assertion is trivial. 
Up to replacing with $v/\|v\|_{\infty}$ under $v\not\equiv0$, we may assume $0\leq v\leq1$ ${\m}$-a.e. on $D$. 
Take an open set $W$ with $W\Subset D$ and ${\rm supp}[v]\subset W$. 
Then, there exists a sequence $v_n\in\mathscr{C}_W$ such that 
$v_n\to v$ in $\mathscr{E}_1^{1/2}$-norm.  
By taking a subsequence if necessary, we may assume $\lim_{n\to\infty}v_n=\tilde{v}$ for $\mathscr{E}$-q.e.~on $E$. 
Take $\psi\in\mathscr{C}_D$ with $0\leq\psi\leq1$ and $\psi=1$ on $W$ and let $\varphi_{\eps}$ be the contraction function which appears in the definition of the special standard core $\mathscr{C}$. 
We define the function $w_n:=\varphi_{1/n}\circ v_n+\frac{1}{n}\psi$ which belongs to $(\mathscr{C}_{D})_+$. 
Indeed, since $v_n=0$ on $W^c$, $w_n=\frac{\psi}{n}\geq0$ on $W^c$. On the other hand, 
$w_n\geq-\frac{1}{n}+\frac{\psi}{n}=0$ on $W$. Thus $w_n\geq0$ on $D$. 
Moreover, $\psi(x)=0$ implies $w_n(x)=0$, because $\psi(x)=0$ induces $x\notin W$, hence $v_n(x)=0$. 
Furthermore, $\lim_{n\to\infty}w_n=\tilde{v}$ for $\mathscr{E}$-q.e. on $E$ and $\{w_n\}$ is $\mathscr{E}_1$-bounded.
The latter implies the existence of a subsequence $\{n_i\}$ such that the sequence $\{\wh{w}_N\}\subset(\mathscr{C}_D)_+$ of the Ces\`aro means $\wh{w}_N:=\frac{1}{N}\sum_{i=1}^Nw_{n_i} \in (\mathscr{C}_D)_+$ $\mathscr{E}_1$-converges to $v$.
Note that the bound from above is preserved and we still have $\supp[\wh{w}_N] \subset \supp[\psi]$.
Taking open sets $U,V$ with $\supp[\psi]\cup{\rm supp}[v]\subset U\Subset V\Subset D$, we have
\begin{align*}
|\mathscr{E}_{\alpha}(u,v-\wh{w}_N)|\leq& 
\mu_{\langle u_U\rangle}^{(c)}(E)^{\frac12}\mu_{\langle v-\wh{w}_N\rangle}^{(c)}(E)^{\frac12}+
\mathscr{E}(u_V,u_V)^{\frac12}\mathscr{E}(v-\wh{w}_N,v-\wh{w}_N)^{\frac12}
\\
&+2\int_{U\times (E\setminus V)}|\tilde{u}(x)-\tilde{u}(y)|
|v(x)-\wh{w}_N(x)|J({\d} x{\d} y)\\
&+\|\tilde{u}_U\|_{L^2(E;\kappa)}\|v-\wh{w}_N\|_{L^2(E;\kappa)}
+\alpha(u_U,v-\wh{w}_N)_{\m},
\end{align*}
which converges to $0$ as $N \to +\infty$ by the Lebesgue's dominated convergence theorem and $\int_{U\times V^c}|\tilde{u}(x)-\tilde{u}(y)|J({\d} x{\d} y)<\infty$. Here $u_V$ is the function in $D(\mathscr{E})$ (resp.~$D(\mathscr{E})_e$) such that $u=u_V$ $\m$-a.e.~on $V$ if $u\in D(\mathscr{E})^{\diamond}_{\loc}$ (resp.~$u\in D(\mathscr{E})^{\diamond}_{e,\loc}$).  
We conclude that, 
\begin{equation*}
\mathscr{E}_{\alpha}(u,v)=\lim_{N\to\infty}\mathscr{E}_{\alpha}(u,\wh{w}_N)\leq0.
\end{equation*}
\end{proof}
\begin{thm}\label{thm:CharacSubhamonicity*}
Suppose that $u\in D(\mathscr{E})_{\loc}^{\dag}(\subset 
D(\mathscr{E})_{\loc}^{\diamond})$ is $\mathscr{E}_{\alpha}$-subharmonic in $D$ in the sense of Definition~\ref{df:subharmonicity}. 
Then $\mathscr{E}_{\alpha}(u,v)\leq0$ for any $v\in D(\mathscr{E}_D)_c^+$. 
\end{thm}
\begin{proof}[{\bf Proof}.]
Suppose $v\in D(\mathscr{E}_D)_c^+$. 
Take open sets $U,V$ with ${\rm supp}[v]\subset U\Subset V\Subset D$. 
Then $v\in D(\mathscr{E}_U)_c^+$. 
Set $v_n:=v\land n$. Then $v_n\in D(\mathscr{E}_D)_c\cap L^{\infty}(D;{\m})_+$ and $v_n$ converges to $v$ with respect to $\mathscr{E}_1^{1/2}$ as $n\to\infty$ by \cite[Theorem~1.4.2(iii)]{FOT}. 
We have 
\begin{align*}
|\mathscr{E}_{\alpha}(u,v-v_n)|\leq& 
\mu_{\langle u_U\rangle}^{(c)}(E)^{\frac12}\mu_{\langle v-v_n\rangle}^{(c)}(E)^{\frac12}+
\mathscr{E}(u_V,u_V)^{\frac12}\mathscr{E}({v}-v_n,{v}-v_n)^{\frac12}
\\
&+2\int_{U\times (E\setminus V)}|\tilde{u}(x)-\tilde{u}(y)|
|\tilde{v}(x)-\tilde{v}_n(x)-\tilde{v}(y)+\tilde{v}_n(y)|J({\d} x{\d} y)\\
&+\|\tilde{u}_U\|_{L^2(E;\kappa)}\|v-v_n\|_{L^2(E;\kappa)}
+\alpha(u_U,v-v_n)_{\m},\\
&\leq \mu_{\langle u_U\rangle}^{(c)}(E)^{\frac12}\mu_{\langle v-v_n\rangle}^{(c)}(E)^{\frac12}+
\mathscr{E}(u_V,u_V)^{\frac12}
\mathscr{E}(v-v_n,v-v_n)^{\frac12}
\\
&+2\left(\int_{U\times (E\setminus V)}|\tilde{u}(x)-\tilde{u}(y)|^2
J({\d} x{\d} y)\right)^{\frac12}\mathscr{E}(
v-v_n,v-v_n)^{\frac12}\\
&+\|\tilde{u}_U\|_{L^2(E;\kappa)}\|\tilde{v}-\tilde{v}_n\|_{L^2(E;\kappa)}
+\alpha(u_U,v-v_n)_{\m},
\end{align*}
which converges to $0$ as $n \to +\infty$. 
We now conclude 
\begin{align*}
\mathscr{E}_{\alpha}(u,v)=\lim_{n\to\infty}\mathscr{E}_{\alpha}(u,v_n)\leq0.
\end{align*} 
\end{proof}
\begin{remark}
{\rm 
\begin{enumerate}
\item[{\rm (1)}] The definition of $\mathscr{E}_{\alpha}$-(sub)harmonicity on $D$ treated in \cite[Definition~2.6]{Chen} or  \cite[Definition~2.6]{CK} is 
a bit restrictive since it requires local boundedness of $u$ on $D$.
\item[{\rm (2)}]
Our definition of $\mathscr{E}_{\alpha}$-subharmonicity on $D$ is a slightly wider notion than the 
$\mathscr{E}_{\alpha}$-subharmonicity treated in \cite{HuaKellerLenzSchmidt}. 
In particular, our $\mathscr{E}_{\alpha}$-subharmonicity (in $D$) covers the notion of 
$\mathscr{E}_{\alpha}$-subharmonicity treated in \cite[Definition~5.3]{MasamuneUemura2} (see Corollary~\ref{cor:newcriterion} above).
\end{enumerate}
}
\end{remark}
Finally, we give a sufficient condition to 
$u\in D(\mathscr{E})^{\diamond}_{\loc}$ for $u\in D(\mathscr{E})_{\loc}$. 
For this, we need the following lemmas. 
\begin{lem}\label{lem:jumpkernelequivalence}
Let ${\sf d}$ be a distance function consistent with the given topology on 
 $E$. 
Fix $o\in E$, $p_1,p_2>0$ and take $U\Subset V\Subset E$. 
\begin{enumerate}
\item[\rm(1)] There exists $C_2=C_2(o,\, p_2,U,V)>0$ such that for any $x\in U$ and $y\notin V$,
\begin{align*}
\frac{1}{{\sf d}(x,y)^{p_2}}\leq C_2\left(1\land\frac{1}{{\sf d}(y,o)^{p_2}} \right).
\end{align*}
\item[\rm(2)] Suppose $o\in V$. Then there exists $C_1=C_1(o,\, p_1,U,V)>0$ such that for any $x\in U$ and $y\notin V$
\begin{align*}
\frac{1}{{\sf d}(x,y)^{p_1}}\geq C_1\left(1\land\frac{1}{{\sf d}(y,o)^{p_1}} \right).
\end{align*}
\end{enumerate}
\end{lem}
\begin{proof}[{\bf Proof.}]
(1): From
\begin{align*}
{\sf d}(y,o)^{p_2}&\leq({\sf d}(x,y)+{\sf d}(x,o))^{p_2}\\
&\leq(2^{p_2-1}\lor1)({\sf d}(x,y)^{p_2}+{\sf d}(x,o)^{p_2})\\
&\leq (2^{p_2-1}\lor1)\left({\sf d}(x,y)^{p_2}+\sup_{z\in U}{\sf d}(z,o)^{p_2}\frac{{\sf d}(x,y)^{p_2}}{{\sf d}(U,V^c)^{p_2}} \right)\\
&\leq (2^{p_2-1}\lor1)\left(1+\frac{\sup_{z\in U}{\sf d}(z,o)^{p_2}}{{\sf d}(U,V^c)^{p_2}} \right){\sf d}(x,y)^{p_2},
\end{align*}
we see 
\begin{align}
\frac{1}{{\sf d}(x,y)^{p_2}}\leq (2^{p_2-1}\lor1)\left(1+\frac{\sup_{z\in U}{\sf d}(z,o)^{p_2}}{{\sf d}(U,V^c)^{p_2}} \right)\frac{1}{{\sf d}(y,o)^{p_2}}.\label{eq:Eq1}
\end{align} 
On the other hand, since ${\sf d}(x,y)^{p_2}\geq{\sf d}(U,V^c)^{p_2}$, 
\begin{align}
\frac{1}{{\sf d}(x,y)^{p_2}}\leq\frac{1}{{\sf d}(U,V^c)^{p_2}}.\label{eq:Eq2}
\end{align}
Combining \eqref{eq:Eq1} and \eqref{eq:Eq2}, we have
\begin{align*}
\frac{1}{{\sf d}(x,y)^{p_2}}&\leq\left\{(2^{p_2-1}\lor1)\left(1+\frac{\sup_{z\in U}{\sf d}(z,o)^{p_2}}{{\sf d}(U,V^c)^{p_2}} \right)\lor \frac{1}{{\sf d}(U,V^c)^{p_2}}
 \right\}\left(1\land\frac{1}{{\sf d}(y,o)^{p_2}} \right)\\
 &\leq C_2\left(1\land\frac{1}{{\sf d}(y,o)^{p_2}} \right).
\end{align*} 
(2): From 
\begin{align*}
{\sf d}(x,y)^{p_1}&\leq (2^{p_1-1}\lor1)({\sf d}(y,o)^{p_1}+{\sf d}(x,o)^{p_1})\\
&\leq (2^{p_1-1}\lor1)\left({\sf d}(y,o)^{p_1}+\sup_{z\in U}{\sf d}(z,o)^{p_1}\frac{{\sf d}(y,o)^{p_1}}{
{\sf d}(o,V^c)^{p_1}} \right)\quad (\because o\in V)\\&=(2^{p_1-1}\lor1)\left(1+\frac{\sup_{z\in U}{\sf d}(z,o)^{p_1}}{{\sf d}(o,V^c)^{p_1}} \right){\sf d}(y,o)^{p_1},
\end{align*}
we see 
\begin{align*}
\frac{1}{{\sf d}(x,y)^{p_1}}&\geq \frac{1}{(2^{p_1-1}\lor1)\left(1+\frac{\sup_{z\in U}{\sf d}(z,o)^{p_1}}{{\sf d}(o,V^c)^{p_1}}  \right)}\cdot\frac{1}{{\sf d}(y,o)^{p_1}}\\
&\geq \frac{{\sf d}(V^c,o)^{p_1}}{(2^{p_1-1}\lor1)
({\sf d}(V^c,o)^{p_1}+\sup_{z\in U}{\sf d}(z,o)^{p_1})
}\left(1\land\frac{1}{{\sf d}(y,o)^{p_1}} \right)\\
&=C_1\left(1\land\frac{1}{{\sf d}(y,o)^{p_1}} \right).
\end{align*}
\end{proof}
\begin{lem}\label{lem:intgrabilityequivalence1}
Let 
${\sf d}$ be a distance function consistent with the given topology on 
 $E$.
Suppose that the jumping measure $J$ satisfies 
$$
\frac{D_1}{{\sf d}(x,y)^{p_1}}{\m}({\d} x)\m({\d} y)\leq
J({\d} x{\d} y)\leq \frac{D_2}{{\sf d}(x,y)^{p_2}}{\m}({\d} x){\m}({\d} y)
$$ 
for some $D_1,D_2>0$ and $0<p_1\leq p_2<\infty$. 
Assume $1\land\frac{1}{{\sf d}(\cdot,o)^{p_1}}\in L^1(E;{\m})$ 
for a fixed $o\in E$ and take $u\in L^1_{\loc}(E;{\m})$. Then the following are equivalent. 
\begin{enumerate}
\item[\rm(1)] $\left(1\land \frac{1}{{\sf d}(\cdot,o)^{p_1}} \right)|u|\in L^1(E;{\m})$.
\item[\rm(2)] For any $U\Subset V\Subset E$ with $o\in V$, 
\begin{align}
\int_{U\times V^c}|u(x)-u(y)|J({\d} x{\d} y)\label{eq:1finite}
<\infty.
\end{align} 
\end{enumerate}
Assume $1\land\frac{1}{{\sf d}(\cdot,o)^{p_1}}\in L^2(E;{\m})$ and 
take $u\in L^2_{\loc}(E;{\m})$. Then the following are equivalent. 
\begin{enumerate}
\item[\rm(3)] $\left(1\land \frac{1}{{\sf d}(\cdot,o)^{p_1}} \right)|u|^2\in L^1(E;{\m})$.
\item[\rm(4)] For any $U\Subset V\Subset E$ with $o\in V$, 
\begin{align}
\int_{U\times V^c}|u(x)-u(y)|^2J({\d} x{\d} y)\label{eq:2finite}
<\infty.
\end{align} 
\end{enumerate}
\end{lem}
\begin{cor}\label{cor:intgrabilityequivalence}
Let ${\sf d}$ be a distance function consistent with the given topology on 
 $E$.
Suppose that the jumping measure $J$ satisfies $\frac{D_1}{{\sf d}(x,y)^{p_1}}{\m}({\d} x){\m}({\d} y)\leq 
J({\d} x{\d} y)\leq \frac{D_2}{{\sf d}(x,y)^{p_2}}{\m}({\d} x){\m}({\d} y)$ for some $D_1,D_2>0$ and $0<p_1\leq p_2<\infty$.
Assume $1\land\frac{1}{{\sf d}(\cdot,o)^{p_1}}\in L^1(E;{\m})$ for a fixed $o\in E$. 
Then the following hold. 
\begin{enumerate}
\item[\rm(1)]
$u\in D(\mathscr{E})^{\diamond}_{\loc}$ implies 
$\left(1\land \frac{1}{{\sf d}(\cdot,o)^{p_1}} \right)|u|\in L^1(E;{\m})$.
\item[\rm(2)] Further suppose that 
$1\land\frac{1}{{\sf d}(\cdot,o)^{p_1}}\in L^1(E;{\m})$ holds for any $o\in E$. 
If $u\in D(\mathscr{E})_{\loc}$ and  
$\left(1\land \frac{1}{{\sf d}(\cdot,o)^{p_1}} \right)|u|\in L^1(E;{\m})$ holds for any $o\in E$, then $u\in D(\mathscr{E})^{\diamond}_{\loc}$.
\end{enumerate}
Assume $1\land\frac{1}{{\sf d}(\cdot,o)^{p_1}}\in L^2(E;{\m})$ 
for a fixed $o\in E$. 
Then the following hold. 
\begin{enumerate}
\item[\rm(3)]
$u\in D(\mathscr{E})^{\dag}_{\loc}$ implies 
$\left(1\land \frac{1}{{\sf d}(\cdot,o)^{p_1}} \right)|u|^2\in L^1(E;{\m})$.
\item[\rm(4)] For $u\in D(\mathscr{E})_{\loc}$, 
$\left(1\land \frac{1}{{\sf d}(\cdot,o)^{p_1}} \right)|u|^2\in L^1(E;{\m})$ implies $u\in D(\mathscr{E})^{\dag}_{\loc}$.
\end{enumerate}
\end{cor}
\begin{proof}[{\bf Proof.}]
It is easy to prove (1) and (3). First we prove (2). Suppose 
$u\in D(\mathscr{E})_{\loc}$. 
Take any $U\Subset V\Subset E$. Then there exists a point $o\in V$. 
For such $o\in V$, $1\land\frac{1}{{\sf d}(\cdot,o)^{p_1}}\in L^1(E;{\m})$ holds by assumption. Then  
$\left(1\land \frac{1}{{\sf d}(\cdot,o)^{p_1}} \right)|u|\in L^1(E;{\m})$ 
implies \eqref{eq:1finite} by Lemma~\ref{lem:intgrabilityequivalence1}. 
Next we prove (4). Suppose 
$u\in D(\mathscr{E})_{\loc}$. 
Assume that $1\land\frac{1}{{\sf d}(\cdot,o)^{p_1}}\in L^2(E;{\m})$ holds for a fixed $o\in E$. 
Take any $U\Subset V\Subset E$ with $o\in V$. 
Then $\left(1\land \frac{1}{{\sf d}(\cdot,o)^{p_1}} \right)|u|^2
\in L^1(E;{\m})$ 
implies \eqref{eq:2finite} by Lemma~\ref{lem:intgrabilityequivalence1}. 
When $o\notin V$, we take a relatively compact open set $W$ such that  $U\Subset V\Subset W\Subset E$ with $o\in W$. Then we have 
\begin{align*}
\int_{U\times W^c}|u(x)-u(y)|^2J({\d} x{\d} y)<\infty.
\end{align*}
From this, 
\begin{align*}
\int_{U\times V^c}|u(x)-u(y)|^2J({\d} x{\d} y)&=
\int_{U\times W^c}|u(x)-u(y)|^2J({\d} x{\d} y)\\
&\hspace{1cm}+
\int_{U\times (W\setminus V)}|u(x)-u(y)|^2J({\d} x{\d} y)\\
&=
\int_{U\times W^c}|u(x)-u(y)|^2J({\d} x{\d} y)\\
&\hspace{1cm}+
\int_{U\times (W\setminus V)}|u_W(x)-u_W(y)|^2J({\d} x{\d} y)<\infty.
\end{align*}
\end{proof}
\begin{proof}[\bf Proof of Lemma~\ref{lem:intgrabilityequivalence1}]
Since $0<p_1\leq p_2<\infty$, we see    
$1\land\frac{1}{{\sf d}(\cdot,o)^{p_1}}\in L^1(E;{\m})$  
(resp. $1\land\frac{1}{{\sf d}(\cdot,o)^{p_1}}\in L^2(E;{\m})$) 
implies $1\land\frac{1}{{\sf d}(\cdot,o)^{p_2}}\in L^1(E;{\m})$ 
(resp.~$1\land\frac{1}{{\sf d}(\cdot,o)^{p_2}}\in L^2(E;{\m})$). 

(1)$\Longrightarrow$(2): 
Assume $1\land\frac{1}{{\sf d}(\cdot,o)^{p_1}}\in L^1(E;{\m})$. 
Suppose (1). Take any 
$U\Subset V\Subset E$ with $o\in V$. 
By Lemma~\ref{lem:jumpkernelequivalence}(1), 
\begin{align*}
\int_{U\times V^c}&|u(x)-u(y)|J({\d} x{\d} y)
\\
&\leq D_2\int_U|u(x)|\int_{V^c}
\frac{{\m}({\d} y)}{{\sf d}(x,y)^{p_2}}{\m}({\d} x)
+D_2\int_U\int_{V^c}\frac{|u(y)|{\m}({\d} y)}{{\sf d}(x,y)^{p_2}}{\m}({\d} x)\\
&\leq D_2C_2\left[\int_U|u|{\d}{\m}\int_{V^c}\left(1\land\frac{1}{{\sf d}(y,o)^{p_2}} \right){\m}({\d} y)\right.
\\
&\hspace{2cm}
+\left.{\m}(U)
\int_{V^c}\left(1\land\frac{1}{{\sf d}(y,o)^{p_2}} \right)|u(y)|{\m}({\d} y)\right]\\
&\leq D_2C_2\left[\int_U|u|{\d}{\m}\int_{E}\left(1\land\frac{1}{{\sf d}(y,o)^{p_2}} \right){\m}({\d} y)\right.
\\
&\hspace{2cm}
+\left.{\m}(U)
\int_{E}\left(1\land\frac{1}{{\sf d}(y,o)^{p_2}} \right)|u(y)|{\m}({\d} y)\right]
<\infty.
\end{align*}
Then we have (2).

(2)$\Longrightarrow$(1): Suppose (2).
Take $U\Subset V\Subset E$ with $o\in V$. 
 By Lemma~\ref{lem:jumpkernelequivalence}(1)(2), 
 \begin{align*}
 D_1C_1{\m}(U)\int_{V^c}&\left(1\land \frac{1}{{\sf d}(y,o)^{p_1}} \right)|u(y)|{\m}({\d} y)\\&=D_1C_1\int_U{\m}({\d} x)\int_{V^c}\left(1\land \frac{1}{{\sf d}(y,o)^{p_1}} \right)|u(y)|{\m}({\d} y)\\
 &\leq D_1\int_{U\times V^c}\frac{|u(y)|}{{\sf d}(x,y)^{p_1}}{\m}({\d} x)
 {\m}({\d} y)\leq\int_{U\times V^c}|u(y)|J({\d} x{\d} y)\\
 &\leq \int_{U\times V^c}|u(x)|J({\d} x{\d} y)+\int_{U\times V^c}|u(y)-u(x)|J({\d} x{\d} y)\\
 &\leq D_2C_2\int_{U}|u(x)|{\m}({\d} x)\int_{V^c}\left(1\land \frac{1}{{\sf d}(y,o)^{p_2}} \right){\m}({\d} y)\\
 &\hspace{1cm}+\int_{U\times V^c}|u(y)-u(x)|J({\d} x{\d} y)
 \\
 &\leq D_2C_2\int_{U}|u(x)|{\m}({\d} x)\int_E\left(1\land \frac{1}{{\sf d}(y,o)^{p_2}} \right){\m}({\d} y)\\
 &\hspace{1cm}+\int_{U\times V^c}|u(y)-u(x)|J({\d} x{\d} y)
  <\infty.
 \end{align*}
 On the other hand, 
 \begin{align*}
 \int_V\left(1\land \frac{1}{{\sf d}(y,o)^{p_1}} \right)|u(y)|{\m}({\d} y)\leq\int_V|u(y)|{\m}({\d} y)<\infty.
 \end{align*}
 Therefore we have (1). The equivalence (3)$\Longleftrightarrow$(4) under 
 $1\land\frac{1}{{\sf d}(\cdot,o)^{p_1}}\in L^2(E;{\m})$
 can be 
 similarly proved. We omit it. 
\end{proof}
The following lemma can be similarly proved by use of Lemma~\ref{lem:jumpkernelequivalence}. We omit its proof. 
\begin{lem}\label{lem:intgrabilityequivalence2}
Let ${\sf d}$ be a distance function consistent with the given topology on 
 $E$.
Suppose that the jumping measure $J$ satisfies 
$\frac{D_1}{{\sf d}(x,y)^{p_1}}{\m}({\d} x){\m}({\d} y)\leq
J({\d} x{\d} y)\leq \frac{D_2}{{\sf d}(x,y)^{p_2}}{\m}({\d} x){\m}({\d} y)$ for some $D_1,D_2>0$ and $0<p_1\leq p_2<\infty$. 
Fix a point $o\in E$. Take $u\in L^1_{\loc}(E;{\m})$. Then the following are equivalent. 
\begin{enumerate}
\item[\rm(1)] $\left(1\land \frac{1}{{\sf d}(\cdot,o)^{p_1}} \right)|u|\in L^1(E;{\m})$.
\item[\rm(2)] For any $U\Subset V\Subset E$ with $o\in V$, 
\begin{align*}
\int_{U\times V^c}|u(y)|J({\d} x{\d} y)
<\infty.
\end{align*} 
\end{enumerate}
Take $u\in L^2_{\loc}(E;{\m})$. Then the following are equivalent. 
\begin{enumerate}
\item[\rm(3)] $\left(1\land \frac{1}{{\sf d}(\cdot,o)^{p_1}} \right)|u|^2\in L^1(E;{\m})$.
\item[\rm(4)] For any $U\Subset V\Subset E$  with $o\in V$, 
\begin{align*}
\int_{U\times V^c}|u(y)|^2J({\d} x{\d} y)
<\infty.
\end{align*} 
\end{enumerate}
\end{lem}
\section{
Main Results}\label{sec:subharmonicity}
A set $B\subset E_{\partial}:=E\cup\{\partial\}$ is called {\it nearly Borel measurable} if for each probability measure $\mu$ on $E_{\partial}$ there exist Borel sets $B_1,B_2\in\mathscr{B}(E_{\partial})$ such that 
$B_1\subset B\subset B_2$ and $\mathbb{P}_{\mu}(X_t\in B_2\setminus B_1\text{ for some }t\geq0)=1$. It is easy to see that nearly Borel measurable sets form a $\sigma$-algebra $\mathscr{B}^n(E_{\partial})$ over $E_{\partial}$. We set   
$\mathscr{B}^n(E):=\{B(\subset E)\mid B\in \mathscr{B}^n(E_{\partial})\}$.
A subset $B$ of $E_{\partial}$ is said to be {\it ${\bf X}$-invariant} 
if $B$ is nearly Borel and
\begin{align*}
\mathbb{P}_x(X_t\in B_{\partial},X_{t-}\in B_{\partial}\text{ for all }t\geq0)=1\quad\text{ for any }x\in B.
\end{align*}
Denote by $\mathscr{B}^{\nu}(E)$ the $\nu$-completion of $\mathscr{B}(E)$ with respect to $\nu\in\mathscr{P}(E)$, where $\mathscr{P}(E)$ is the family of all probability measures on $E$. A set $B$ 
is called {\it universally measurable} if $B\in\mathscr{B}^*(E):=\bigcap_{\nu\in\mathscr{P}(E)}\mathscr{B}^{\nu}(E)$. 
It is known that $\mathscr{B}^*(E)$ forms a $\sigma$-algebra and any nearly Borel subset of $E$ is universally measurable.  
A subset $N$ of $E$ is called {\it properly exceptional} if 
$N\in\mathscr{B}^n(E)$, ${\m}(N)=0$ and $E\setminus N$ is ${\bf X}$-invariant.  
It is known that for $B\in\mathscr{B}^n(E_{\partial})$, the hitting time $\sigma_B:=\inf\{t>0\mid X_t\in B\}$ is an $(\mathscr{F}_t)_{t\geq0}$-stopping time and $\mathbb{P}_x(\sigma_B=0)=0$ or $1$ in view of the Blumenthal's $0$-$1$ law. We set $B^r:=\{x\in E\mid \mathbb{P}_x(\sigma_B=0)=0\}$ and each point $x\in B^r$ is called a {\it regular point of $B$}. 
A set $A$ is called {\it finely open} if for each $x\in A$ there exists a set 
$B=B(x)\in\mathscr{B}^n(E)$ such that $E\setminus A\subset B$ and 
$x\notin B^r$.
The totality of finely open sets satisfies the Hausdorff axioms of topology and is called the {\it fine topology}, which is finer than the given topology on $E$, because any open set is finely open in view of the right continuity of the sample paths. 
A nearly Borel subset $N$ of $E$ is called {\it ${\m}$-polar} if $\mathbb{P}_{\m}(\sigma_N<\infty)=0$ and any subset $N$ of $E$ is called {\it exceptional} if there exists an ${\m}$-polar set $\tilde{N}$ containing $N$ (see \cite[Chapter IV, Definition~5.17]{MR}). 
For a statement $P(x)$ with respect to $x\in E$, it is called that $P(x)$ holds q.e.~$x\in E$ if $\{x\in E\mid P(x)\text{ fails }\}$ is an exceptional. 
Any properly exceptional set $N$ is exceptional. 
It is shown in \cite[Chapter IV, Theorem~5.29(i)]{MR} that $\mathscr{E}$-exceptionality is 
equivalent to the exceptionality, i.e., \lq\lq$\mathscr{E}$-q.e.\rq\rq\, is 
equivalent to \lq\lq q.e.\rq\rq.  Let $D$ be an open set.           
A function $u$ defined on $\mathscr{E}$-q.e.~on $D$ is said to be {\it q.e.~finely continuous on $D$} if there exists a nearly Borel exceptional set 
$N$ such that $D\setminus N$ is finely open and $u$ is nearly Borel and finely continuous on $D\setminus N$.   
To characterize $\mathscr{E}$-subharmonic functions in $D(\mathscr{E})_{\loc}^{\diamond}$ without assuming local boundedness in terms of {\bf X}, we introduce 
the following stochastic subharmonicity: 
\begin{defn}[{{\bf Sub/Super-harmonicity}}]
\label{df:stochasticsubharmonicity}
{\rm Let $D$ be an open set in $E$. We say that a nearly Borel
$\overline{\R}$-valued function $u$ defined on $E$ is
\emph{subharmonic {\rm(}resp.~superharmonic{\rm)} in $D$} if 
for any open subset $U$ of $D$ with $U\Subset D$, $t\mapsto u(X_{t\wedge\tau_U})$ is a uniformly integrable right continuous $\mathbb{P}_x$-submartingale (resp.~$\mathbb{P}_x$-supermartingale) for q.e.~$x\in E$. A nearly Borel $\overline{\R}$-valued function $u$ on $E$ is said to be
\emph{harmonic in $D$} if $u$ is  both superharmonic and subharmonic in $D$.
If $t\mapsto u(X_{t\wedge\tau_U})$ is a uniformly integrable right
continuous $\mathbb{P}_x$-submartingale for all $x\in E$, $u$ is called \emph{subharmonic in $D$ without exceptional set}. 
The \emph{subharmonicity/harmonicity in $D$ without exceptional set} is
analogously defined.
}
\end{defn}

\begin{defn}[Sub/Super-harmonicity in the weak sense]
{\rm Let $D$ be an open set in $E$. We say that a nearly Borel
$\overline{\R}$-valued function $u$ defined on $E$ is \emph{subharmonic {\rm(}resp.~superharmonic{\rm)} in $D$ in the weak sense} if $u$ is q.e.~finely continuous on $D$ and for any relatively compact open subset $U$ with $U\Subset D$, $\mathbb{E}_x[|{u}|(X_{\tau_U})]<\infty$ for q.e.~$x\in E$ and ${u}(x)\leq\mathbb{E}_x[{u}(X_{\tau_U})]$ (resp.~${u}(x)\geq\mathbb{E}_x[{u}(X_{\tau_U})]$) holds for q.e.~$x\in E$.
A nearly Borel measurable $\overline{\R}$-valued function $u$ on $E$ is said to be \emph{harmonic in $D$ in the weak sense} if $u$ is both superharmonic and subharmonic in $D$ in the weak sense.
If $u$ is finely continuous (nearly) Borel $\overline{\R}$-valued function on $D$ and for any relatively compact open subset $U$ with $U\Subset D$, $\mathbb{E}_x[|{u}|(X_{\tau_U})]<\infty$ for all $x\in E$, and ${u}(x)\leq\mathbb{E}_x[{u}(X_{\tau_U})]$ holds  for all $x\in E$, then $u$ is called \emph{subharmonic in $D$ in the weak sense without exceptional set}.
The \emph{subharmonicity/harmonicity in $D$ in the weak sense without exceptional set} is analogously defined.
}
\end{defn}  

It is easy to see that for subharmonic functions $u_1,u_2$ on $D$ (resp. subharmonic functions $u_1,u_2$ in $D$ in the weak sense), 
$u_1\lor u_2$ is subharmonic in $D$ (resp. subharmonic in $D$ in the weak sense).

\bigskip

We introduce the following condition:
\begin{align}
\left.\begin{array}{cc}\text{For any relatively compact open set $U$ with}& U\Subset D\text{ and } D\setminus \overline{U}\ne\emptyset, \\
\hspace{-2.2cm}
\mathbb{P}_x(\tau_U<\infty)>0\text{ for q.e.~}x\in U. & \end{array}\right.\label{e:2.3}
\end{align}
 Condition
\eqref{e:2.3} is satisfied if 
$(\mathscr{E}_D,D(\mathscr{E}_D))$ is
\emph{irreducible}, that is, any $(P_t^D)$-invariant subset $B$ of $D$ is
trivial in the sense that ${\m}(B)=0$ or ${\m}(D\setminus B)=0$. 
Indeed, set $G:=D\setminus\overline{U}\subset E\setminus U$. Since ${\m}(G)>0$, $G$ is non-exceptional. 
Applying \cite[Theorem~3.5.6(i)]{CFbook} to ${\bf X}_D$, we have 
\begin{align*}
\mathbb{P}_x(\sigma_G<\tau_D)>0\quad\text{ q.e.} \quad x\in D. 
\end{align*}
Since $\tau_U=\sigma_{E\setminus U}\land\zeta\leq\sigma_G$, 
we obtain
\begin{align*}
\mathbb{P}_x(\tau_U<\infty)\geq
\mathbb{P}_x(\tau_U<\tau_D)
\geq \mathbb{P}_x(\sigma_G<\tau_D)>0 \quad\text{ q.e.} \quad x\in D.
\end{align*}
Note that ${\bf X}_U$ is transient under \eqref{e:2.3}. 
Indeed, let $U=U^{(c)}+U^{(d)}+N$ be an ergodic decomposition of 
${\bf X}_U$ (see \cite{Kw:ergodic} for ergodic decomposition). If ${\m}(U^{(c)})>0$, then 
\begin{align*}
\mathbb{P}_x(\tau_U<\infty)\leq\mathbb{P}_x(\tau_{U^{(c)}}<\infty)=0\quad\text{ for \ \ q.e.}\quad x\in U^{(c)},
\end{align*}
which contradicts \eqref{e:2.3}, hence ${\m}(U^{(c)})=0$, i.e., ${\bf X}_U$ is transient. Moreover, we introduce the following condition: 
For an open set $D$, we consider the following condition for a (nearly) Borel $\overline{\R}$-valued function $u\in D(\mathscr{E})_{e,\loc}$ 
on $E$: 
\begin{align}
\left\{\begin{array}{rl}\1_U\mathbb{E}_{\cdot}[|\tilde{u}-\tilde{u}_V|(X_{\tau_U})]\in D(\mathscr{E}_U)_e&\text{\!\!for any open sets $U, V$ with $U\Subset V\Subset D$} \\\text{and some\!\!\!}&\text{$u_V\in D(\mathscr{E})_e$ satisfying 
$u=u_V$ ${\m}$-a.e.~on $V$.}\end{array}\right.
\label{eq:2.3}
\end{align}

The condition \eqref{eq:2.3} is weaker than  
the condition as in \cite[(2.5)]{Chen} or \cite[(2.3)]{CK}. 
Indeed, we have the following:

\begin{prop}\label{prop:Sufficient}
Take $u\in D(\mathscr{E})_{e,\loc}$ and fix an open set $D$. Suppose one of the following holds: 
\begin{enumerate}
\item[{\rm(1)}] ${\bf X}$ is a diffusion process. 
\item[{\rm(2)}] $u\in L^{\infty}(D;{\m})$.
\item[{\rm(3)}] $u\in D(\mathscr{E})_e$. 
\item[{\rm(4)}] $u\in L^{\infty}_{\loc}(D;{\m})$ and 
for any relatively compact 
open sets $U,V$ with $U\Subset V\Subset D$,
there exists $\phi_V\in\mathscr{C}_D$ such that $0\leq\phi_V\leq1$ on $D$, $\phi_V=1$ on $V$ and 
$\1_U\mathbb{E}_{\cdot}[(1-\phi_V)|\tilde{u}|(X_{\tau_U})]\in D(\mathscr{E}_U)_e$. 
\item[{\rm (5)}] $u\in D(\mathscr{E})_{\loc}$, there exists an intrinsic metric ${\sf d}$ in the sense of 
\cite[Definition~4.1]{FLW} such that all open balls with respect to ${\sf d}$ are relatively compact, and 
for any relatively compact open sets $U,V$ with $U\Subset V\Subset D$,
there exists ${\sf d}$-Lipschitz function $\phi_V$ with 
compact support in $D$ such that  $0\leq\phi_V\leq1$ on $D$, $\phi_V=1$ on $V$ and 
$\mathbb{E}_{\cdot}[(1-\phi_V)|\tilde{u}|(X_{\tau_U})]\in D(\mathscr{E}_U)_e$. 
\end{enumerate}
Then \eqref{eq:2.3} holds for $u$. 
\end{prop}
\begin{proof}[{\bf Proof.}]
Under the condition (1), we see $X_{\tau_U}\in\partial U\cup \{\partial\}$ $\mathbb{P}_x$-a.s., for $x\in U$ which means 
$|\tilde{u}-\tilde{u}_V|(X_{\tau_U})=0$ $\mathbb{P}_x$-a.s. 
So \eqref{eq:2.3} trivially holds.
Under the condition (2), we can take $u_V\in D(\mathscr{E})_e\cap L^{\infty}(E;{\m})$ with $\|u_V\|_{\infty}\leq\|u\|_{\infty}$, hence 
$u-u_V$ is bounded. Then we obtain $\1_U\mathbb{E}_{\cdot}[|\tilde{u}-\tilde{u}_V|(X_{\tau_U})]\in D(\mathscr{E}_U)_e$ in the same way of the proof of \cite[Lemma~2.3]{Chen}. Under the condition  (3), we can take $u_V=u\in D(\mathscr{E})_e$, then 
\eqref{eq:2.3} trivially holds. Under the condition (4), thanks to the boundedness of $u$ on $V$, $u_V:=\phi_Vu\in D(\mathscr{E}_D)$ does the job. Under the condition (5), for any ${\sf d}$-Lipschitz function $\phi$ with compact support, $\phi u\in D(\mathscr{E})$ without assuming the local boundedness of $u$. Indeed, 
by \cite[Theorem~4.9]{FLW}, we have 
$\mu^{c}_{\langle \phi u\rangle}\leq 2\phi^2\mu^{c}_{\langle  u\rangle}+2u^2\mu^{c}_{\langle \phi\rangle}\leq
2\phi^2\mu^{c}_{\langle  u\rangle}+2({\rm Lip}(\phi))^2u^2{\m}$ and   
$\mu^{j}_{\langle \phi u\rangle}({\d} x)\leq 2 \tilde{u}(x)^2
\mu^{j}_{\langle \phi\rangle}({\d} x)
+2\phi(x)^2\mu^{j}_{\langle u\rangle}({\d} x)\leq
2({\rm Lip}(\phi))^2\1_{{\rm supp}[\phi]}(x)u(x)^2{\m}({\d} x)+2\phi(x)^2\mu^{j}_{\langle u\rangle}({\d} x)$. Moreover, $\mu^{\kappa}_{\langle \phi u\rangle}({\d} x)\leq\|\phi\|_{\infty}^2 \mu^{\kappa}_{\langle u\rangle}({\d} x)$. Then we obtain $\phi u\in D(\mathscr{E})$. 
Applying this to the ${\sf d}$-Lipschitz function $\phi_V$ specified in 
(5), $u_V:=\phi_Vu\in  D(\mathscr{E})$ does the job. 
\end{proof}

Now we give another sufficient condition to \eqref{eq:2.3}. 
\begin{prop}\label{prop:sufficient{eq:2.3}}
Let ${\sf d}$ be a distance function consistent with the given topology on 
 $E$.
Suppose that the L\'evy system $(N,H)$ satisfies  
$N(x,{\d} y)\leq D_2\frac{{\m}({\d} y)}{{\sf d}(x,y)^p}$ for some $D_2>0$ and  $p>0$ and $H_t=t$. 
Let $u$ be a Borel function satisfying $\left(1\land \frac{1}{{\sf d}(\cdot,o)^p} \right)|u|^q\in L^1(E;{\m})$ with $q\in]0,+\infty[$ and any point $o\in E$.  
Suppose that $\sup_{x\in U}\E_x[\tau_U]<\infty$ for any $U\Subset D$. 
Then, for any $U\Subset V\Subset D$ with $o\in V$, 
\begin{align}
\sup_{x\in U}\E_x[\1_{V^c}|u|^q(X_{\tau_U})]<\infty.\label{eq:1}
\end{align}
From this, for a nearly Borel function $u\in D(\mathscr{E})_{\loc}$ 
satisfying $\left(1\land \frac{1}{{\sf d}(\cdot,o)^p} \right)|u|\in L^1(E;{\m})$ {\rm(}resp.~$\left(1\land \frac{1}{{\sf d}(\cdot,o)^p} \right)|u|^2\in L^1(E;{\m})${\rm)}, we have 
\begin{align}
\sup_{x\in U}\E_x[|u-u_V|(X_{\tau_U})]<\infty\quad 
\text{{\rm(}resp.~$\sup_{x\in U}\E_x[|u-u_V|^2(X_{\tau_U})]<\infty${\rm)}}\label{eq:Finiteness}
\end{align}
for any $U\Subset V\Subset E$ and $u_V\in D(\mathscr{E})$ satisfying $u=u_V$ ${\m}$-a.e.~on $V$, in particular, \eqref{eq:2.3} for $u\in D(\mathscr{E})^{\diamond}_{\loc}$ holds, 
provided $\sup_{x\in U}\E_x[\tau_U]<\infty$ for any $U\Subset D$.
\end{prop}
\begin{proof}[{\bf Proof.}]
By use of L\'evy system, Lemma~\ref{lem:jumpkernelequivalence}(1) and 
\cite[(A.3.33)]{CFbook}, 
\begin{align*}
\sup_{x\in U}\E_x[(\1_{V^c}|u|^q)(X_{\tau_U})]&=
\sup_{x\in U}\E_x[(\1_{V^c}|u|^q)(X_{\tau_U})-\1_{V^c}|u|^q(X_{\tau_U-}))]\\
&=\sup_{x\in U}\E_x\left[\sum_{s\leq\tau_U}((\1_{V^c}|u|^q)(X_s)-(\1_{V^c}|u|^q)(X_{s-}))\right]\\
&=\sup_{x\in U}\E_x\left[\int_0^{\tau_U}\left(\int_E\left((
\1_{V^c}|u|^q)(y)-(\1_{V^c}|u|^q)(X_s)\right)N(X_s,{\d} y)  \right){\d} H_s \right]\\
&=\sup_{x\in U}\E_x\left[\int_0^{\tau_U}\left(\int_E(\1_{V^c}|u|^q)(y)N(X_s,{\d} y)  \right){\d} H_s \right]\\
&\leq D_2\sup_{x\in U}\E_x\left[\int_0^{\tau_U}
\int_{V^c}|u|^q(y)\frac{{\m}({\d} y)}{{\sf d}(X_s,y)^p}{\d} H_s
\right]
\\
&\leq D_2C_2\sup_{x\in U}\E_x\left[\int_0^{\tau_U}\int_{V^c}|u|^q(y)\left(
1\land \frac{1}{{\sf d}(y,o)^p}
 \right){\m}({\d} y){\d} s \right]\\
 &=D_2C_2\int_{V^c}\left(
1\land \frac{1}{{\sf d}(y,o)^p}
 \right)|u(y)|^q{\m}({\d} y)\cdot\sup_{x\in U}\E_x[\tau_U]\\
 &\leq D_2C_2\int_{E}\left(
1\land \frac{1}{{\sf d}(y,o)^p}
 \right)|u(y)|^q{\m}({\d} y)\cdot\sup_{x\in U}\E_x[\tau_U]<\infty. 
\end{align*}
Next, for a given nearly Borel function $u\in  D(\mathscr{E})_{\loc}$ satisfying $\left(1\land\frac{1}{{\sf d}(\cdot,o)^p} \right)|u|\in L^1(E;{\m})$ (resp.~$\left(1\land\frac{1}{{\sf d}(\cdot,o)^p} \right)|u|^2\in L^1(E;{\m})$), 
we prove \eqref{eq:Finiteness} for any $U\Subset V\Subset E$ with $o\in V$. 
Since $u\in D(\mathscr{E})_{\loc}$, 
there exists $u_V\in D(\mathscr{E})\subset L^2(E;{\m})\subset L^1_{\loc}(E;{\m})$ such that $u=u_V$ ${\m}$-a.e.~on $V$.
Applying Lemma~\ref{lem:intgrabilityequivalence1} to $u_V\in L^2(E;{\m})\subset L^1_{\loc}(E;{\m})$, we have 
\begin{align*}
\left(1\land \frac{1}{{\sf d}(\cdot,o)^p} \right)|u_V|\in L^1(E;{\m}) \qquad 
\text{{\rm(}resp.~$\left(1\land \frac{1}{{\sf d}(\cdot,o)^p} \right)|u_V|^2\in L^1(E;{\m})$ {\rm)}} 
\end{align*}
under $u\in D(\mathscr{E})_{\loc}$.
Then the first argument of the present proof yields 
\begin{align}
\sup_{x\in U}\E_x[\1_{V^c}|u_V|(X_{\tau_U})]<\infty\qquad 
\text{{\rm(}resp.~$\sup_{x\in U}\E_x[\1_{V^c}|u_V|^2(X_{\tau_U})]<\infty${\rm)}}
\label{eq:2}
\end{align}
under $u\in D(\mathscr{E})_{\loc}$. Combining \eqref{eq:1} and \eqref{eq:2}, we obtain \eqref{eq:Finiteness}.
Finally, \eqref{eq:2.3} can be deduced 
from \eqref{eq:Finiteness} 
by the same way of the proof of  \cite[Lemma~2.3]{Chen}.
\end{proof}
A universally measurable function 
$u$ defined on an open set $U$ is said to be 
{\it excessive with respect to ${\bf X}_U$} if it is non-negative on $U$ and 
$p_t^Uu(x)\uparrow u(x)$ as $t\downarrow0$ for any $x\in U$. 
Here $p_t^Uu(x):=\mathbb{E}_x[u(X_t):t<\tau_U]$. For $u:=\1_U\mathbb{E}_{\cdot}[h(X_{\tau_U})]$ with a non-negative (nearly) Borel measurable function 
$h$ on $E$, $u$ is always excessive with respect to ${\bf X}_U$. Indeed, 
\begin{align*}
p_t^Uu(x)&=p_t^U(\1_U\mathbb{E}_{\cdot}[h(X_{\tau_U})])\\
&=\mathbb{E}_x[\1_U(X_t)\mathbb{E}_{X_t}[h(X_{\tau_U})]:t<\tau_U]\\
&=\mathbb{E}_x[\1_U(X_t)\mathbb{E}_x[h(X_{t+\tau_U\circ\theta_t})\,|\,\mathscr{F}_t]:t<\tau_U]\\
&=\mathbb{E}_x[\mathbb{E}_x[h(X_{t+\tau_U\circ\theta_t})\1_{\{t<\tau_U\}}\,|\,\mathscr{F}_t]]\\
&=\mathbb{E}_x[h(X_{\tau_U}):t<\tau_U]\\
&=\1_U(x)\mathbb{E}_x[h(X_{\tau_U}):t<\tau_U]\\
&\uparrow \1_U(x)\mathbb{E}_x[h(X_{\tau_U})]=u(x)\quad\text{ as }\quad t\downarrow0.
\end{align*} 
\begin{lem}\label{lem:lattice1}
If $u_1,u_2\in  D(\mathscr{E})^{\diamond}_{e,\loc}$ satisfy \eqref{eq:2.3}, then $u_1\land u_2$ and $u_1\lor u_2$ also satisfy \eqref{eq:2.3}. 
\end{lem}
\begin{proof}[{\bf Proof.}]
By elementary inequalities $|a\land b-c\land d|\leq |a-c|+|b-d|$ and $|a\lor b-c\lor d|\leq |a-c|+|b-d|$, we have 
\begin{align*}
\mathbb{E}_{\cdot}[|\tilde{u}_1\land \tilde{u}_2-\widetilde{(u_1\land u_2)}_V|(X_{\tau_U})]&=
\mathbb{E}_x[|\tilde{u}_1\land \tilde{u}_2-\widetilde{(u_1)}_V\land \widetilde{(u_2)}_V|(X_{\tau_U})]\\&\leq\mathbb{E}_x[|\tilde{u}_1-\widetilde{(u_1)}_V|(X_{\tau_U})]+\mathbb{E}_x[|\tilde{u}_2-\widetilde{(u_2)}_V|(X_{\tau_U})],\\
\mathbb{E}_{\cdot}[|\tilde{u}_1\lor \tilde{u}_2-\widetilde{(u_1\lor u_2)}_V|(X_{\tau_U})]&=
\mathbb{E}_x[|\tilde{u}_1\lor \tilde{u}_2-\widetilde{(u_1)}_V\lor \widetilde{(u_2)}_V|(X_{\tau_U})]\\&\leq\mathbb{E}_x[|\tilde{u}_1-\widetilde{(u_1)}_V|(X_{\tau_U})]+\mathbb{E}_x[|\tilde{u}_2-\widetilde{(u_2)}_V|(X_{\tau_U})].
\end{align*}
Note that there exists $(u_i)_V\in D(\mathscr{E})$ such that 
$u_i=(u_i)_V$ ${\m}$-a.e.~on $V$ for each $i=1,2$. This means 
$u_1\land u_2=(u_1)_V\land (u_2)_V$ and 
(resp.~$u_1\lor u_2=(u_1)_V\lor (u_2)_V$)   
${\m}$-a.e.~on $V$, consequently 
$\wt{(u_1\land u_2)_V}=\wt{(u_1)_V}\land\wt{(u_2)_V}$ (resp.~$\wt{(u_1\lor u_2)_V}=\wt{(u_1)_V}\lor\wt{(u_2)_V}$) $\mathscr{E}$-q.e.~on $V$. 
Since $\1_U\mathbb{E}_{\cdot}[|\tilde{u}_1\land \tilde{u}_2-\widetilde{(u_1\land u_2)}_V|(X_{\tau_U})]$ and $\1_U\mathbb{E}_{\cdot}[|\tilde{u}_1\lor \tilde{u}_2-\widetilde{(u_1\lor u_2)}_V|(X_{\tau_U})]$ are 
excessive with respect to ${\bf X}_U$, we can obtain the assertion 
by \cite[Lemma~3.3]{CK}.
\end{proof}
\begin{lem}\label{lem:lattice2}
If $u\in  D(\mathscr{E})^{\diamond}_{e,\loc}$ satisfies \eqref{eq:2.3}, 
then for $U\Subset D$
\begin{align}
\mathbb{E}_{\cdot}[\tilde{u}(X_{\tau_U})]-u
\in D(\mathscr{E}_U)_e.\label{eq:2.3*}
\end{align}
\end{lem}
\begin{proof}[{\bf Proof.}]
Note that $\1_U\mathbb{E}_{\cdot}[(\tilde{u}-\tilde{u}_V)_{\pm}(X_{\tau_U})]$ is excessive with respect to ${\bf X}_U$. Indeed, 
\begin{align*}
p_t^U(\1_U\mathbb{E}_{\cdot}[(\tilde{u}-\tilde{u}_V)_{\pm}(X_{\tau_U})])(x)&=
\mathbb{E}_x[\mathbb{E}_{X_t}[(\tilde{u}-\tilde{u}_V)_{\pm}(X_{\tau_U})]:t<\tau_U]\\
&=\mathbb{E}_x[\mathbb{E}_x[
(\tilde{u}-\tilde{u}_V)_{\pm}(X_{t+\tau_U\circ \theta_t})
\,|\,\mathscr{F}_t]:t<\tau_U]\\
&=\mathbb{E}_x[\mathbb{E}_x[(\tilde{u}-\tilde{u}_V)_{\pm}(X_{t+\tau_U\circ \theta_t})\1_{\{t<\tau_U\}}\,|\,\mathscr{F}_t]]\\
&=\mathbb{E}_x[\mathbb{E}_x[(\tilde{u}-\tilde{u}_V)_{\pm}(X_{\tau_U})\1_{\{t<\tau_U\}}\,|\,\mathscr{F}_t]]\\
&=\mathbb{E}_x[(\tilde{u}-\tilde{u}_V)_{\pm}(X_{\tau_U})\1_{\{t<\tau_U\}}]
\\&\uparrow \1_U(x)\mathbb{E}_x[(\tilde{u}-\tilde{u}_V)_{\pm}(X_{\tau_U})]
\quad\text{ as }\quad t\downarrow 0.
\end{align*}
Then we have 
$\1_U\mathbb{E}_{\cdot}[(\tilde{u}-\tilde{u}_V)_{\pm}(X_{\tau_U})]\in D(\mathscr{E}_U)_e$ by \cite[Lemma~3.3]{CK}, hence $\1_U\mathbb{E}_{\cdot}[(\tilde{u}-\tilde{u}_V)(X_{\tau_U})]\in D(\mathscr{E}_U)_e$. On the other hand, since 
$u_V\in D(\mathscr{E})_e$,  
$\mathbb{E}_{\cdot}[\tilde{u}_V(X_{\tau_U})]
\in D(\mathscr{E})_e$. Therefore, 
$\mathbb{E}_{\cdot}[\tilde{u}(X_{\tau_U})]-u=
\mathbb{E}_{\cdot}[\tilde{u}_V(X_{\tau_U})]-u_V
+\1_U \mathbb{E}_{\cdot}[(\tilde{u}-\tilde{u}_V)(X_{\tau_U})]\in D(\mathscr{E}_U)_e$. 
\end{proof}
\begin{cor}\label{cor:lattice2}
Fix $\alpha>0$. 
If $u\in  D(\mathscr{E})^{\diamond}_{\loc}$ satisfies 
$\1_U(\cdot)\mathbb{E}_{\cdot}[e^{-\alpha\tau_U}|\tilde{u}-\tilde{u}_V|(X_{\tau_U})]\in D(\mathscr{E})$ for $U\Subset V\Subset D$ with $u_V\in D(\mathscr{E})$ satisfying $u=u_V$ ${\m}$-a.e.~on $V$, then for any $U\Subset D$,
\begin{align}
\mathbb{E}_{\cdot}[e^{-\alpha\tau_U}\tilde{u}(X_{\tau_U})]-u
\in D(\mathscr{E}_U).\label{eq:2.3**}
\end{align}
\end{cor}
\begin{proof}[{\bf Proof.}]
The proof is similar to the proof of Lemma~\ref{lem:lattice2}. We omit it. 
\end{proof}
\begin{lem}\label{lem:Subharmonicity}
If $u\in  D(\mathscr{E})^{\diamond}_{e,\loc}$ satisfies \eqref{eq:2.3}, then 
$\mathbb{E}_{\cdot}[\tilde{u}(X_{\tau_U})]\in D(\mathscr{E})^{\diamond}_{e,\loc}$ is $\mathscr{E}$-harmonic in $U$ for $U\Subset D$. 
\end{lem}
\begin{proof}[{\bf Proof.}]
Take $V$ with $U\Subset V\Subset D$. 
We set $h_1:=\mathbb{E}_{\cdot}[\tilde{u}_V(X_{\tau_U})]$ and $h_2:=\mathbb{E}_{\cdot}[(\tilde{u}-\tilde{u}_V)(X_{\tau_U})]$, which is well-defined by condition \eqref{e:2.3}. Since $h_1\in D(\mathscr{E})_e$ is $\mathscr{E}$-harmonic in $U$, it suffices to show the 
$\mathscr{E}$-harmonicity on $U$ of $h_2$. We already know $h_2
=\1_Uh_2+(u-u_V)\in D(\mathscr{E}_U)_e+D(\mathscr{E})^{\diamond}_{e,\loc}\subset D(\mathscr{E})^{\diamond}_{e,\loc}$. From this,  
$\mathbb{E}_{\cdot}[\tilde{u}(X_{\tau_U})]=h_1+h_2\in D(\mathscr{E})_{e,\loc}^{\diamond}$.  
By use of the L\'evy system of ${\bf X}$ and \cite[(A.3.33)]{CFbook}, we have  
\begin{align}
\1_U(x)&\mathbb{E}_x[(\tilde{u}-\tilde{u}_V)^{\pm}(X_{\tau_U})]\notag\\&=
\1_U(x)\mathbb{E}_x\left[\sum_{s\leq\tau_U}\left((\tilde{u}-\tilde{u}_V)^{\pm}(X_s)-(\tilde{u}-\tilde{u}_V)^{\pm}(X_{s-})\right)\right]\notag
\\&=\1_U(x)\mathbb{E}_x\left[\int_0^{\tau_U}\left(\int_{V^c}\left((\tilde{u}-\tilde{u}_V)^{\pm}(y)-(\tilde{u}-\tilde{u}_V)^{\pm}(X_s)\right)N(X_s,{\d} y)\right){\d} H_s \right]\notag\\
&=\1_U(x)\mathbb{E}_x\left[\int_0^{\tau_U}\left(\int_{V^c}(\tilde{u}-\tilde{u}_V)^{\pm}(y)N(X_s,{\d} y)\right){\d} H_s \right],\label{eq:relation}
\end{align}
which is excessive with respect to ${\bf X}_U$, hence 
belongs to $D(\mathscr{E}_U)_e$ by 
\cite[Lemma~3.3]{CK}.
Define measures $\mu_1,\mu_2$ on $U$ by 
\begin{align*}
\mu_1({\d} x):&=\1_U(x)\left(\int_{V^c}(\tilde{u}-\tilde{u}_V)^+(y)N(x,{\d} y)\right)\mu_H({\d} x),\\ 
\mu_2({\d} x):&=\1_U(x)\left(\int_{V^c}(\tilde{u}-\tilde{u}_V)^-(y)N(x,{\d} y)\right)\mu_H({\d} x).
\end{align*}
Then, they are smooth measures with respect to 
${\bf X}_U$. Indeed, we have 
\begin{align*}
\mu_1(U)&=\int_U\int_{V^c}(\tilde{u}-\tilde{u}_V)^+(y)N(x,{\d} y)\mu_H({\d} x)\\
&=\int_U\int_{V^c}|(\tilde{u}-\tilde{u}_V)^+(y)-(\tilde{u}-\tilde{u}_V)^+(x)|N(x,{\d} y)\mu_H({\d} x)\\
&=\frac12\int_{U\times V^c}|(\tilde{u}-\tilde{u}_V)^+(y)-(\tilde{u}-\tilde{u}_V)^+(x)|J({\d} x{\d} y)
<\infty
\end{align*}
by Lemma~\ref{lem:contraction} with $u-u_V\in D(\mathscr{E})^{\diamond}_{e,\loc}$. Since $\mu_H$ charges no $\mathscr{E}$-exceptional set, 
$\mu_1$ does so. 
This together with $\mu_1(U)<\infty$ means that $\mu_1$ is smooth with respect to 
${\bf X}_U$. Similarly, so is $\mu_2$. 
For a smooth measure $\nu$ of ${\bf X}_U$, we write  
$R^U\!\nu:=\mathbb{E}_{\cdot}[A_{\tau_U}^{\nu}]$, where $A^{\nu}$ is the PCAF of ${\bf X}_U$ with Revuz measure $\nu$.
Then $R^U\!\mu_1(x)=\1_U\mathbb{E}_{\cdot}[(\tilde{u}-\tilde{u}_V)^+(X_{\tau_U})]\in D(\mathscr{E}_U)_e$ and $R^U\!\mu_2(x)=\1_U\mathbb{E}_{\cdot}[(\tilde{u}-\tilde{u}_V)^-(X_{\tau_U})]\in D(\mathscr{E}_U)_e$
by the calculation \eqref{eq:relation} above.  
For $k\geq1$, let $F_k:=\{x\in U\mid R^U\!\mu_1(x)\leq k\}$, which is a finely closed subset of $U$. Define $\nu_k:=\1_{F_k}\mu_1$. 
Then for $x\in F_k$, $R^U\!\nu_k(x)\leq R^U\!\mu_i(x)\leq k$, while for $x\in U\setminus F_k$, 
\begin{align*}
R^U\!\nu_k(x)=\mathbb{E}_x[R^U\!\nu_k(X_{\sigma_{F_k}})]\leq k.
\end{align*}
That is, we have $R^U\!\nu_k\leq k\land R^U\!\mu_1$. As both
 $R^U\!\nu_k$ and $k\land R^U\!\mu_1$ are excessive functions of ${\bf X}_U$ and $\m(U)<\infty$, we have $\{R^U\!\nu_k,k\land R^U\!\mu_1\}\subset D(\mathscr{E}_U)$ and 
 \begin{align*}
 \mathscr{E}(R^U\!\nu_k,R^U\!\nu_k)\leq\mathscr{E}(k\land R^U\!\mu_1,k\land R^U\!\mu_1)\leq \mathscr{E}(R^U\!\mu_1,R^U\!\mu_1)<\infty
 \end{align*}
by \cite[Theorem~1.1.5 and Lemma~1.2.3]{CFbook}. Moreover, for each $k\geq1$, we have by \cite[Theorem~4.1.1]{CFbook} or \cite[Theorem~5.1.3]{FOT}
\begin{align*}
\mathscr{E}(R^U\!\nu_k,R^U\!\nu_k)&=\uparrow\lim_{\beta\to\infty}\beta((I-\beta R^U_{\beta})R^U\!\nu_k,R^U\!\nu_k)_{L^2(U;{\m})}\\
&=\uparrow\lim_{\beta\to\infty}
\beta (R^U_{\beta}\nu_k,R^U\!\nu_k)_{L^2(U;{\m})}\\
&=\uparrow\lim_{\beta\to\infty}
\langle \nu_k,\beta R^U_{\beta}R^U\!\nu_k\rangle\\
&=\langle \nu_k\,R^U\!\nu_k\rangle,
\end{align*}
which increases to $\langle \mu_1\,R^U\!\mu_1\rangle$. 
Consequently, $\langle \mu_1\,R^U\!\mu_1\rangle\leq \mathscr{E}(R^U\!\mu_1,R^U\!\mu_1)<\infty$. Similarly, $\langle \mu_2\,R^U\!\mu_2\rangle\leq \mathscr{E}(R^U\!\mu_2,R^U\!\mu_2)<\infty$.  
Applying \cite[Lemma~2.4]{Chen} to ${\bf X}_U$, we have 
\begin{align*}
\mathscr{E}(R^U\!\mu_i,\varphi)=
\langle \mu_i,\varphi\rangle
\end{align*}
for any $\varphi\in D(\mathscr{E}_U)\cap C_c(U)$ and $i=1,2$. 
Then we have that for any $\varphi\in \mathscr{C}_U$
\begin{align*}
\mathscr{E}(\1_Uh_2,\varphi)&=\mathscr{E}(R^U\!\mu_1-R^U\!\mu_2,\varphi)
\\&=\langle \mu_1-\mu_2,\varphi\rangle
\\&=\int_E\varphi(x)\1_U(x)\int_{V^c}
(\tilde{u}-\tilde{u}_V)(y)N(x,{\d} y)\mu_H({\d} x)\\
&=2\int_E\varphi(x)\1_U(x)\int_{V^c}
(\tilde{u}-\tilde{u}_V)(y)J({\d} x{\d} y).
\end{align*}
On the other hand, for $\varphi\in\mathscr{C}_U$
\begin{align*}
\mathscr{E}(u-u_V,\varphi)&=\int_{E\times E}(\tilde{u}(x)-\tilde{u}_V(x)-\tilde{u}(y)+\tilde{u}_V(y))
(\varphi(x)-\varphi(y))J({\d} x{\d} y)\\
&=2\int_{U\times V^c}(-\tilde{u}(y)+\tilde{u}_V(y))
\varphi(x)J({\d} x{\d} y)\\
&=-2\int_{U\times V^c}(\tilde{u}(y)-\tilde{u}_V(y))\varphi(x)J({\d} x{\d} y)
\end{align*}
Hence
\begin{align*}
\mathscr{E}(h_2,\varphi)&=\mathscr{E}(\1_Uh_2,\varphi)+\mathscr{E}(u-u_V,\varphi)=0.
\end{align*}
\end{proof}
\begin{cor}\label{cor:Subharmonicity}
Fix $\alpha>0$ and suppose \eqref{e:2.3}.
If $u\in  D(\mathscr{E})^{\diamond}_{\loc}$ satisfies 
$\1_U(\cdot)\mathbb{E}_{\cdot}[e^{-\alpha\tau_U}|\tilde{u}-\tilde{u}_V|(X_{\tau_U})]\in D(\mathscr{E})$ for $U\Subset V\Subset D$ with $u_V\in D(\mathscr{E})$ satisfying $u=u_V$ ${\m}$-a.e.~on $V$, then $\mathbb{E}_{\cdot}[e^{-\alpha\tau_U}\tilde{u}(X_{\tau_U})]\in D(\mathscr{E})^{\diamond}_{\loc}$ is $\mathscr{E}_{\alpha}$-harmonic in $U$ for any $U\Subset D$. 
\end{cor}
\begin{proof}[{\bf Proof.}]
The proof is similar to the proof of Lemma~\ref{lem:Subharmonicity}.
Set $h_1^{\alpha}:=\mathbb{E}_{\cdot}[e^{-\alpha\tau_U}\tilde{u}_V(X_{\tau_U})]$ and $h_2^{\alpha}:=\mathbb{E}_{\cdot}[e^{-\alpha\tau_U}(\tilde{u}-\tilde{u}_V)(X_{\tau_U})]$. Since $h_1^{\alpha}\in D(\mathscr{E})$ is $\mathscr{E}_{\alpha}$-harmonic in $U$, it suffices to show the 
$\mathscr{E}_{\alpha}$-harmonicity on $U$ of $h_2^{\alpha}$. We already know $h_2^{\alpha}
=\1_Uh_2^{\alpha}+(u-u_V)\in D(\mathscr{E}_U)+D(\mathscr{E})^{\diamond}_{\loc}\subset D(\mathscr{E})^{\diamond}_{\loc}$. 
From this,  
$\mathbb{E}_{\cdot}[e^{-\alpha\tau_U}\tilde{u}(X_{\tau_U})]=h_1^{\alpha}+h_2^{\alpha}\in D(\mathscr{E})_{\loc}^{\diamond}$.
 
By \cite[(A.3.33)]{CFbook}, instead of \eqref{eq:relation} we have  
\begin{align}
\1_U(x)&\mathbb{E}_x[e^{-\alpha\tau_U}(\tilde{u}-\tilde{u}_V)^{\pm}(X_{\tau_U})]\notag\\&=\1_U(x)\mathbb{E}_x\left[\int_0^{\tau_U}e^{-\alpha t}\left(\int_{V^c}(\tilde{u}-\tilde{u}_V)^{\pm}(y)N(X_s,{\d} y)\right){\d} H_s \right].\label{eq:relation2}
\end{align}
As in the proof of Lemma~\ref{lem:Subharmonicity}, the measure $\mu_i$ becomes an $\alpha$-order energy measure of finite integrals for each $i=1,2$. Then we can deduce $R_{\alpha}^U\mu_i\in D(\mathscr{E}_U)$ and 
$\mathscr{E}_{\alpha}(R_{\alpha}^U\mu_i,\varphi)=\langle \mu_i,\varphi\rangle$ 
for $\varphi\in\mathscr{C}_U$ and $i=1,2$. The rest of the proof is similar to the proof of Lemma~\ref{lem:Subharmonicity}. 
\end{proof}
The ${\m}$-symmetric Markov process {\bf X} is said to satisfy the \emph{absolute continuity condition with respect to ${\m}$
{\rm(}{\bf (AC)} in short{\rm)}} if $p_t(x,\cdot)\ll{\m}(\cdot)$ for any $t>0$ and $x\in E$, where $p_t(x,A):=\P_x(X_t\in A)$ for $A\in\mathscr{B}(E)$.  
The next lemma relaxes the local square integrability condition for $u$ in 
\cite[Lemma~3.9]{CK} without assuming the local boundedness.

\begin{lem}[{cf.~\cite[Lemma~3.9]{CK}}]\label{lem:EquivalenceSubharmonicity}
Let $D$ be an open set and $u$ is a nearly Borel $\overline{\R}$-valued function on $E$. 
Then we have the following:
\begin{enumerate}
\item[\rm(1)]
Suppose \eqref{e:2.3} and $u\in L^1_{\loc}(D;{\m})$. If $u$ is subharmonic in $D$, then $u$ is subharmonic in $D$ in the weak sense. Moreover, under ${\bf (AC)}$ for ${\bf X}$, if $u$ is a subharmonic function on $D$ without  exceptional set and is in $L^1_{\loc}(D;{\m})$, then $u$ is subharmonic in $D$ in the weak sense without exceptional set.
\item[\rm(2)] If $u$ is a nearly Borel finely continuous 
$\overline{\R}$-valued function such that $u$ is subharmonic in $D$ in the weak sense, 
then for any $U\Subset D$, $t\mapsto u(X_{t\land\tau_U})$ is a 
{\rm(}not necessarily uniformly integrable{\rm)} $\mathbb{P}_x$-submartingale for q.e.~$x\in E$. Moreover, under ${\bf (AC)}$ for ${\bf X}$ , if $u$ is a finely continuous subharmonic function on $D$ in the weak sense without exceptional set, then for $U$ above, $\{u(X_{t\land\tau_U})\}_{t\geq0}$ is a {\rm(}not necessarily uniformly integrable{\rm)} $\P_x$-submartingale for all $x\in U$. 
\end{enumerate}
\end{lem}
\begin{proof}[{\bf Proof.}]
The statement (2) is nothing but \cite[Lemma~3.9(ii)]{CK}. 
So we only prove (1). Suppose $u\in L^1_{\loc}(D;{\m})$ is subharmonic 
and take $U\Subset D$. 
We see 
that $\{u(X_{t\land\tau_U})\}_{t\geq0}$ is a uniformly integrable $\P_x$-submartingale for q.e.~$x\in E$. Then as $t\to\infty$, $u(X_{t\land\tau_U})$ converges in $L^1(\P_x)$ as well as $\P_x$-a.s. to some random variable $\xi$ for q.e.~$x\in E$. Set $Y_t:=u(X_{t\land\tau_U})$ for $t\in[0,+\infty[$ and $Y_{\infty}:=\xi$. Then $\{Y_t\}_{t\in[0,\infty]}$ is a right-closed $\P_x$-submartingale for q.e.~$x\in E$. Applying optional sampling theorem to $\{Y_t\}_{t\in[0,+\infty]}$, we have $\E_x[u(X_{\tau_U})]<\infty$ and $u(x)\leq\E_x[Y_{\tau_U}]$ for q.e.~$x\in E$. Note that $Y_{\tau_U}\1_{\{\tau_U<\infty\}}=u(X_{\tau_U})$ and 
$Y_{\tau_U}=u(X_{\tau_U})+\xi\1_{\{\tau_U=\infty\}}$ $\P_x$-a.s. for q.e.~$x\in E$. Set $u_2(x):= \E_x[\xi\1_{\{\tau_U=\infty\}}]$. 
We show $u_2=0$ q.e.~on $E$ under \eqref{e:2.3}. We can confirm that for each $t>0$, $p_t^Uu_2(x)=u_2(x)$ for q.e.~$x\in U$. 
Indeed, 
\begin{align*}
p_t^Uu_2(x)&=\mathbb{E}_x[\mathbb{E}_{X_t}[\xi\1_{\{\tau_U=\infty\}}]:t<\tau_U]\\
&=\mathbb{E}_x[\mathbb{E}_x[\xi\circ\theta_t
\1_{\{\tau_U\circ\theta_t=\infty\}}\,|\,\mathscr{F}_t]:t<\tau_U]\\
&=\mathbb{E}_x[\mathbb{E}_x[\xi\circ\theta_t\1_{\{t+\tau_U\circ\theta_t=\infty,t<\tau_U\}}\,|\,\mathscr{F}_t]]\\
&=\mathbb{E}_x[\xi\1_{\{\tau_U=\infty\}}]=u_2(x).
\end{align*}
Note that 
\begin{align*}
u_2(x)=\lim_{t\to\infty}\E_x[u(X_t)\1_U(X_t)\1_{\{\tau_U=\infty\}}]
\end{align*}
for q.e.~$x\in E$. Then
\begin{align*}
\int_U|u_2|{\m}({\d} x)\leq\varliminf_{n\to\infty}\int_Up_n\1_U|u|(x){\m}({\d} x)\leq\int_U|u(x)|{\m}({\d} x)<\infty.
\end{align*}
Thus $u_2\in L^1(U;{\m})$. Moreover, 
\begin{align*}
\int_U|u_2|{\d}{\m}&=\int_U|p_t^Uu_2|{\d}{\m}\\
&\leq \int_Up_t^U|u_2|{\d}{\m}\\
&=\int_U(p_t^U1)|u_2|{\d}{\m}=\int_U\P_x(t<\tau_U)|u_2(x)|{\m}({\d} x)\\
&\hspace{1cm}\downarrow\int_U(1-\P_x(\tau_U<\infty))|u_2(x)|{\m}({\d} x)\quad \text{ as }\quad t\uparrow\infty.
\end{align*}
From this with $u_2\in L^1(U;{\m})$, 
\begin{align*}
\int_U\P_x(\tau_U<\infty)|u_2(x)|{\m}({\d} x)=0.
\end{align*}
By \eqref{e:2.3}, $u_2=0$ ${\m}$-a.e.~on $U$, hence $u_2=0$ q.e.~on $U$, because $u_2=p_t^Uu_2$ q.e.~on $U$ is $\mathscr{E}_U$-quasi-continuous 
on $U$.   
Therefore, we obtain that $u(x)\leq\E_x[u(X_{\tau_U})]$ for q.e.~$x\in U$ 
under  \eqref{e:2.3}. That is, $u$ is subharmonic in $D$ in the weak sense. 
The rest of the proof is similar to that in \cite[Lemma~3.9(i)]{CK}. We omit it. 
\end{proof}
Next theorem extends \cite[Theorem~3.10]{CK} without assuming the local boundedness.  

\begin{thm}\label{thm:3.10}
Let $D$ be an open set. 
Suppose that $u\in D(\mathscr{E})_{e,\loc}^{\diamond}$ satisfies \eqref{eq:2.3}.  
If $u$ is $\mathscr{E}$-subharmonic in $D$, then $u$ is subharmonic in $D$ in the sense of Definition~\ref{df:stochasticsubharmonicity}.  
Moreover, under {\bf (AC)}, if $u$ is finely continuous and {\rm(}nearly{\rm)} Borel measurable, then $u$ is subharmonic in $D$ without exceptional set.
\end{thm}
\begin{proof}[{\bf Proof.}]
Suppose that $u\in D(\mathscr{E})_{e,\loc}^{\diamond}$ satisfying \eqref{eq:2.3} is $\mathscr{E}$-subharmonic in $D$.
 Fix an open set $U \Subset D$.
Set $h_0(x):=\mathbb{E}_x[{u}(X_{\tau_U})]$, then, $u-h_0\in D(\mathscr{E}_U)_e$ by Lemma~\ref{lem:lattice2}. 
By Lemma~\ref{lem:Subharmonicity}, $h_0$ is $\mathscr{E}$-harmonic in $U$, which implies the 
$\mathscr{E}$-subharmonicity of $(u-h_0)^+\in D(\mathscr{E}_U)_e$ 
in $U$ thanks to \cite[Lemma~3.2]{CK}. 
As in the proof of \cite[Theorem~3.10]{CK}, we can obtain 
\begin{align*}
\mathscr{E}((u-h_0)^+,(u-h_0)^+)=0.
\end{align*} 
Thus by \cite[Lemma~2.2]{Chen}, we get that 
$(u-h_0)^+(X_t)=(u-h_0)^+(x)$ for all $t\geq0$ $\mathbb{P}_x$-a.s. for q.e.~$x\in E$. Consequently, $(u-h_0)^+(X_t)$ is a bounded $\mathbb{P}_x$-martingale for q.e.~$x\in E$, hence, the sets $A := \{u > h_0\}$ and $A^c = \{u\leq  h_0\}$ are ${\bf X}$-invariant. 
Thus, after removing a properly exceptional set of ${\bf X}$ if necessary, we may and do assume that $h_0$ is finely continuous and that either $A = E$ or $A^c = E$ by regarding $A$ or $A^c$ as the whole space.  
The rest of the proof is similar to \cite[proof of Theorem~3.10]{CK}. 
We emphasize that the local boundedness of $u$ on $D$ is not used in this proof. We omit the proof under {\bf (AC)}.
\end{proof}
For the proof of main theorems below, we need the time change method. 
Let $g\in L^1(E;{\m})$ be a function satisfying $0<g\leq1$ ${\m}$-a.e. 
Consider a PCAF $A_t^g:=\int_0^tg(X_s){\d} s$ and set its right continuous inverse by $\tau_t:=\inf\{s>0\mid t<A_s^g\}$. The Markov process 
 $\check{\bf X}:=(\Omega,\check{X}_t,\mathbb{P}_x)$ defined by $\check{X}_t:=X_{\tau_t}$ is called the {\it time changed process}. 
 It is known that $\check{\bf X}$ is $g{\m}$-symmetric (see \cite[Theorem~6.2.1]{FOT}) and 
 the measure $g{\m}$ has full quasi support (see \cite[p.~190]{FOT}). 
 The Dirichlet form 
 on $L^2(E;g{\m})$ associated with $\check{\bf X}$ is denoted by
 $(\check{\mathscr{E}}, D(\check{\mathscr{E}}))$. 
 Since $g{\m}$ hs full quasi support, 
 $D(\mathscr{E})_e= D(\check{\mathscr{E}})_e$ and 
 $\mathscr{E}(u,v)=\check{\mathscr{E}}(u,v)$ for $u,v\in D(\mathscr{E})_e= D(\check{\mathscr{E}})_e$ (see \cite[(6.2.23)]{FOT}). 
 Next lemma is fundamental. 
 \begin{lem}\label{lem:TimeChange}
 We consider a function $g\in L^1(E;{\m})$ with $0<g\leq 1$ ${\m}$-a.e.~on $E$. Let $\check{\bf X}$ be the time changed process  by $A_t^g$ and 
 $(\check{\mathscr{E}},D(\check{\mathscr{E}}))$ on $L^2(E;g{\m})$ the  Dirichlet form associated with $\check{\bf X}$.   
 The following hold. 
 \begin{enumerate}
 \item[\rm(1)] For $u\in D(\mathscr{E})_{e,\loc}$ {\rm(}resp.~$u\in D(\mathscr{E})_{e,\loc}^{\dag}$, $u\in D(\mathscr{E})_{e,\loc}^{\diamond}${\rm)}, there exists 
 $g\in L^1(E;{\m})$ satisfying $0<g\leq 1$ ${\m}$-a.e.~such that 
 $u\in D(\check{\mathscr{E}})_{\loc}$ {\rm(}resp.~$u\in D(\check{\mathscr{E}})_{\loc}^{\dag}$, $u\in D(\check{\mathscr{E}})_{\loc}^{\diamond}${\rm)} holds for the Dirichlet form $(\check{\mathscr{E}}, D(\check{\mathscr{E}}))$ on $L^2(E;g{\m})$ associated with the time changed process $\check{\bf X}$. In particular, for $u\in D(\mathscr{E})_{e,\loc}^{\diamond}$, $\mathscr{E}$-subharmonicity of $u$ in $D$ is equivalent to the 
 $\check{\mathscr{E}}$-subharmonicity of $u$ in $D$.  
 \item[\rm(2)] For a nearly Borel $\overline{\R}$-valued function $u$, 
 the subharmonicity of $u$ in $D$ with respect to ${\bf X}$ is equivalent to 
 the subharmonicity of $u$ in $D$ with respect to $\check{\bf X}$.
 \item[\rm(3)] For a nearly Borel $\overline{\R}$-valued function $u$, 
 the subharmonicity of $u$ in $D$ in the weak sense with respect to ${\bf X}$ is equivalent to  the subharmonicity of $u$ in $D$ in the weak sense with respect to $\check{\bf X}$.
 \end{enumerate}
 \end{lem}
 \begin{proof}[\bf Proof.]
 \begin{enumerate}
 \item[(1)] It suffices to show the case $u\in D(\mathscr{E})_{e,\loc}$. 
 Let $\{U_n\}$ be an increasing sequence satisfying $U_n\Subset U_{n+1}\Subset E$ and $\bigcup_{n=1}^{\infty}U_n=E$. 
 Since $u\in D(\mathscr{E})_{e,\loc}$, there exists $u_n\in D(\mathscr{E})_e$ such that $u=u_n$ ${\m}$-a.e.~on $U_n$. Fix $f\in L^1(E;{\m})$ with $0<f\leq 1$ ${\m}$-a.e.~and set $g:=f/(\sup_{n\geq1}u_n^2\lor 1)$. 
 Such an $f$ always exists thanks to the $\sigma$-finiteness of ${\m}$. Then we see 
 $\sup_{n\geq1}\int_Eu_n^2g{\d}{\m}\leq\int_Ef{\d}{\m}<\infty$, in particular, 
 $u_n\in D(\mathscr{E})_e\cap L^2(E;g{\m})=D(\check{\mathscr{E}})_e\cap L^2(E;g{\m})=D(\check{\mathscr{E}})$. For any $U\Subset E$, we can take 
 $U_n$ such that $\overline{U}\subset U_n$. This means $u\in D(\check{\mathscr{E}})_{\loc}$. 
 \item[(2)] Consider the first exit time 
 $\check{\tau}_U:=A_{\tau_U}^g$ from $U$ with respect to $\check{\bf X}$. 
 For $\check{t}=A_t^g$, we see $\check{X}_{\check{t}\land \check{\tau}_U}=X_{t\land \tau_U}$. Then $u(\check{X}_{\check{t}\land \check{\tau}_U})=u(X_{t\land \tau_U})$.
 Moreover, the natural filtration $\check{\mathscr{F}}^0$ with respect to $\check{\bf X}$ is nothing but the natural filtration 
  $\mathscr{F}_t^0$ with respect to ${\bf X}$:
 $\check{\mathscr{F}}^0:=\sigma(\check{X}_{\check{u}}:\check{u}\leq\check{t})
 =\sigma(X_u:u\leq t)=\mathscr{F}_t^0$, because $\check{X}_{\check{t}}=X_{\tau_{\check{t}}}=X_{\tau_{A_t^g}}=X_t$.  
 The $(\check{\mathscr{F}}_t)$-submartingale property with respect to $\check{\bf X}$ of $u(\check{X}_{\check{t}\land \check{\tau}_U})$ is also equivalent to the $({\mathscr{F}}_t)$-submartingale property with respect to ${\bf X}$ of $u(X_{t\land \tau_U})$. Here $(\check{\mathscr{F}}_t)$ is the minimal augmented filtration with respect to $\check{\bf X}$. 
 It is easy to see the equivalence of uniform integrability between 
 $\{u(X_{t\land\tau_U})\}_{t\geq0}$ and $\{u(\check{X}_{\check{t}\land\check{\tau}_U})\}_{\check{t}\geq0}$. 
 \item[(3)] It is easy to see $\check{X}_{\check{\tau}_U}=X_{\tau_U}$. 
 This means the conclusion. 
\end{enumerate}
 \end{proof}
Next theorem is proved in \cite[Theorem~3.11]{CK}. 
\begin{thm}[{\cite[Theorem~3.11]{CK}}]\label{thm:Domain}
Let $D$ be an open set and $u\in L^{\infty}_{\loc}(D;{\m})$. Suppose one of the following holds:
\begin{enumerate}
\item[\rm(1)] $u$ is subharmonic in $D$.
\item[\rm(2)] $u$ is subharmonic in $D$ in the weak sense and \eqref{e:2.3} holds. 
\end{enumerate} 
Then $u\in D(\mathscr{E}_D)_{\loc}$. 
\end{thm}
Next theorem extends \cite[Theorem~3.12]{CK} without assuming 
the local boundedness. 

\begin{thm}\label{thm:3.10*}
Let $D$ be an open set and $u\in D(\mathscr{E})_{e,\loc}^{\diamond}$ satisfy 
 \eqref{eq:2.3}. 
Suppose that one of the following holds: 
\begin{enumerate}
\item[\rm(1)] $u$ is subharmonic in $D$.
\item[\rm(2)] $u$ is subharmonic in $D$ in the weak sense  
and \eqref{e:2.3} holds. 
\end{enumerate}
Then $u$ is $\mathscr{E}$-subharmonic in $D$. 
\end{thm}
\begin{proof}[{\bf Proof.}]
By Lemma~\ref{lem:TimeChange}, we may assume that 
$u\in D(\mathscr{E})_{\loc}^{\diamond}\subset L^2_{\loc}(E;{\m})$. 
Recall that ${\bf X}_U$ is transient under \eqref{e:2.3}.  
First we prove the assertion for the case (2). 
For $U\Subset D$, we have 
$\mathbb{E}_x[|u(X_{\tau_U})|]<\infty$ for q.e.~$x\in E$. 
For $U\Subset D$, $u(x)\leq\mathbb{E}_x[u(X_{\tau_U})]$ for q.e.~$x\in U$, which trivially holds for the case (2). 
Since $u\in  D(\mathscr{E})_{e,\loc}^{\diamond}$ satisfies \eqref{eq:2.3}, $\mathbb{E}_{\cdot}[u(X_{\tau_U})]-u\in D(\mathscr{E}_U)_e$ and $h_0:=\mathbb{E}_{\cdot}[u(X_{\tau_U})]\in D(\mathscr{E})^{\diamond}_{e,\loc}$ is $\mathscr{E}$-harmonic in $U$ for $U\Subset D$   
by Lemmas~\ref{lem:lattice2} and \ref{lem:Subharmonicity}.  
Set 
$u_0:=h_0-u\in D(\mathscr{E}_U)_e$ for such $U$. 
Note that $u_0$ is non-negative q.e.~on $U$.  
Take any $O\Subset U\Subset D$.   
By replacing $U$ with $O$, we have $\mathbb{E}_x[u_0(X_{\tau_O})]\leq u_0(x)$ for q.e.~$x\in O$, hence q.e.~on $U$ by taking a properly exceptional set $N$ of ${\bf X}$ and restricting the process ${\bf X}_U$ to ${\bf X}_{U\setminus N}$ if necessary. 
Then $u_0$ is excessive with respect to 
${\bf X}_U$ by use of \cite[Theorem~12.4]{Dyn:Mar}.
The rest of the proof for $\mathscr{E}$-subharmonicity in $D$ is the same as in the proof of 
\cite[Theorem~3.12]{CK} under the transience of ${\bf X}_U$. 
Next we prove the assertion for the case (1). Note that $u$ is automatically q.e.~finely continuous. As in the proof of Lemma~\ref{lem:EquivalenceSubharmonicity}, 
$Y_t:=u(X_{t\land\tau_U})$ for $t\in[0,+\infty[$, $Y_{\infty}:=\xi$ for $t=\infty$ forms a closed $\mathbb{P}_x$-submartingale and $u(x)\leq \mathbb{E}_x[Y_{\tau_U}]=:h_*(x)$. Then 
$Y_{\tau_U}=u(X_{\tau_U})+\xi\1_{\{\tau_U=\infty\}}$ $\mathbb{P}_x$-a.s.~for 
q.e.~$x\in E$. Set $u_2(x):=\mathbb{E}_x[\xi\1_{\{\tau_U=\infty\}}]$. 
Then $h_*(x)=h_0(x)+u_2(x)=\mathbb{E}_x[u(X_{\tau_U})]+u_2(x)$. We now show $u_2\in L^2(U;{\m})$. Indeed, 
\begin{align*}
\int_Uu_2^2{\d}{\m}&\leq \int_U\mathbb{E}_x[|\xi|\1_{\{\tau_U=\infty\}}]^2{\d}{\m}\\
&=\int_U\lim_{n\to\infty}\mathbb{E}_x[|u(X_{n\land\tau_U})|
\1_{\{\tau_U=\infty\}}]^2{\d}{\m}
\\
&\leq\varliminf_{n\to\infty}\int_U\mathbb{E}_x[|u(X_n)|\1_{\{\tau_U=\infty\}}]^2{\d}{\m}
\\
\hspace{4cm}&\leq\varliminf_{n\to\infty}\int_U\mathbb{E}_x[|u(X_n)|\1_{\{n<\tau_U\}}]^2{\d}{\m}
\\
&=\varliminf_{n\to\infty}\int_U(p_n^U(\1_U|u|))^2{\d}{\m}\\
&
=\varliminf_{n\to\infty}\int_Up_n(\1_U|u|^2){\d}{\m}\leq\int_Uu^2{\d}{\m}<\infty.
\end{align*}
On the other hand, we know $p_t^Uu_2=u_2$. This means that 
$u_2\in D(\mathscr{E}_U)$ and $u_2$ is $\mathscr{E}$-harmonic in $U$.
Since $h_*=\mathbb{E}_{\cdot}[u(X_{\tau_U})]+u_2$, $h_*$ is $\mathscr{E}$-harmonic in $U$, because $\mathbb{E}_x[u(X_{\tau_U})]\in D(\mathscr{E})_{e,\loc}^{\diamond}$ is $\mathscr{E}$-harmonic in $U$ as proved above.  
Set $u_0:=h_*-u=\mathbb{E}_{\cdot}[u(X_{\tau_U})]+u_2-u$. Then 
$u_0\in D(\mathscr{E})_{e,\loc}^{\diamond}$ is $\mathscr{E}$-superharmonic in $U$. 
Therefore 
\begin{align*}
\mathscr{E}(u,\phi)&=\mathscr{E}(h_*-u_0,\phi)\\
&=-\mathscr{E}(u_0,\phi)+\mathscr{E}(h_*,\phi)\\
&=-\mathscr{E}(u_0,\phi)\leq0\quad\text{ for }\quad \phi\in D(\mathscr{E}_U)\cap C_c(U)_+.
\end{align*}
Since $U$ is arbitrary, we obtain the $\mathscr{E}$-subharmonicity of $u$ in $D$.
\end{proof}
\begin{remark}
{\rm The proof of Theorem~\ref{thm:3.10*}
fulfills a gap of the proof of \cite[Theorem~3.12]{CK}. 
}
\end{remark}

Combining Lemma~\ref{lem:EquivalenceSubharmonicity}, Theorems~\ref{thm:3.10} and \ref{thm:3.10*}, 
we have the following: 
\begin{thm}\label{thm:main}
Let $D$ be an open set and $u$ a nearly Borel $\overline{\R}$-valued function on $E$. 
Suppose that $u\in D(\mathscr{E})_{e,\loc}^{\diamond}$ satisfies 
\eqref{eq:2.3}. Then the following are equivalent to each other.
\begin{enumerate}
\item[\rm(1)] $u$ is $\mathscr{E}$-subharmonic in $D$. 
\item[\rm(2)] $u$ is subharmonic in $D$.
\end{enumerate}
If we further assume \eqref{e:2.3}, 
 then {\rm(1)} and {\rm(2)} are equivalent to  
\begin{enumerate}
\item[\rm(3)] $u$ is subharmonic in $D$ in the weak sense.
\end{enumerate}
Under {\bf (AC)}, we can replace {\rm(2)} and {\rm(3)}  
with refined statements without exceptional sets, respectively. 
\end{thm}
\begin{proof}[\bf Proof.]
The implications (1)$\Longleftrightarrow$(2)$\Longleftarrow$(3) holds by 
Theorems~\ref{thm:3.10} and \ref{thm:3.10*}.  
The implication (2)$\Longrightarrow$(3) follows from Lemma~\ref{lem:EquivalenceSubharmonicity} under 
\eqref{e:2.3} and $u\in L^1_{\loc}(D;{\m})$. 
It suffices to prove (2)$\Longrightarrow$(3) 
under \eqref{e:2.3} 
without assuming $u\in L^1_{\loc}(D;{\m})$. 
Suppose that $u$ is subharmonic in $D$. 
Note that the condition 
\eqref{e:2.3} is invariant under the time change. 
Indeed,  \eqref{e:2.3} with respect to ${\bf X}$ 
is equivalent to the exceptionality with respect to ${\bf X}$ of 
the set $N:=\{x\in U\mid\mathbb{P}_x(\tau_U=\infty)=1\}$ . 
For $x\in N$, the sample path is always included in $U$. 
Since the shape of the sample does not change under the time change by 
$A_t^g$, we have that 
$N=\{x\in U\mid\mathbb{P}_x(\check{\tau}_U=\infty)=1\}$ is 
an exceptional with respect to $\check{\bf X}$. This means that \eqref{e:2.3} 
holds for $\check{\bf X}$.  The converse assertion also holds. 
For $u\in D(\mathscr{E})_{e,\loc}^{\diamond}$, there exists 
$g\in L^1(E;{\m})$ satisfying $0<g\leq1$ ${\m}$-a.e. such that 
$u\in D(\check{\mathscr{E}})_{\loc}^{\diamond}$ by Lemma~\ref{lem:TimeChange}(1). In particular, 
$u\in L^1_{\loc}(D;g{\m})$. Since $u\in L^1_{\loc}(D;g{\m})$ and \eqref{e:2.3}  holds with respect to $\check{\bf X}$, $u$ is subharmonic in $D$ in the weak sense 
with respect to $\check{\bf X}$ in view of Lemma~\ref{lem:EquivalenceSubharmonicity}.    
Thus, we obtain that $u$ is subharmonic in $D$ in the weak sense  
 with respect to ${\bf X}$ by  Lemma~\ref{lem:TimeChange}(3), i.e., 
 (3) holds. 
\end{proof}
\begin{cor}\label{cor:main0}
Let $D$ be an open set and $u$ a nearly Borel $\overline{\R}$-valued function on $E$. 
Suppose that $u\in D(\mathscr{E})_{\loc}^{\diamond}$ satisfies \eqref{eq:2.3}. Then the same conclusion 
as in Theorem~\ref{thm:main} holds.
\end{cor}
\begin{cor}\label{cor:main}
Let $D$ be an open set and $u$ a nearly Borel $\overline{\R}$-valued function on $E$. 
Assume $J=0$. 
Suppose $u\in D(\mathscr{E})_{\loc}$. 
Then the same conclusion as in Theorem~\ref{thm:main} holds.
Under {\bf (AC)}, we can replace {\rm(2)} and {\rm(3)}
with refined statements without exceptional sets, respectively.
\end{cor}

Next corollary is an improvement of \cite[Corollary~2.10]{CK} without assuming the local boundedness.
\begin{cor}\label{cor:lattice}
Let $D$ be an open subset of $E$. 
\begin{enumerate}
\item[\rm(1)] Let $\eta\in C^1(\R)$ be 
a convex function and $u\in D(\mathscr{E
})_{e,\loc}^{\diamond}$ be an $\mathscr{E}$-harmonic function in $D$ satisfying 
\eqref{eq:2.3}. Suppose that $\eta$ has a bounded first derivative or $u$ is bounded on $E$. Then $\eta(u)\in D(\mathscr{E})_{e,\loc}^{\diamond}$ and is $\mathscr{E}$-subharmonic in $D$ satisfying \eqref{eq:2.3}. If further $u\in D(\mathscr{E
})_{\loc}^{\diamond}$, then $\eta(u)\in D(\mathscr{E})_{\loc}^{\diamond}$.  
\item[\rm(2)] Let $\eta\in C^1(\R)$ be 
an increasing convex function and $u\in D(\mathscr{E
})_{e,\loc}^{\diamond}$ be an $\mathscr{E}$-subharmonic function in $D$ satisfying \eqref{eq:2.3}. Suppose that $\eta$ has a bounded first derivative or $u$ is bounded on $E$. Then the same conclusion of {\rm(1)} also holds.  
\item[\rm(3)] If $u\in D(\mathscr{E})_{e,\loc}^{\diamond}$ {\rm(}resp.~$u\in D(\mathscr{E})_{\loc}^{\diamond}${\rm)} is an $\mathscr{E}$-harmonic function in $D$ satisfying \eqref{eq:2.3}, then  
$|u|\in  D(\mathscr{E})_{e,\loc}^{\diamond}$ {\rm(}resp.~$|u|\in D(\mathscr{E})_{\loc}^{\diamond}${\rm)} 
satisfies \eqref{eq:2.3} and $|u|$ is $\mathscr{E}$-subharmonic in $D$. 
Let $p>1$ and $u\in D(\mathscr{E})_{e,\loc}^{\diamond}$ be an $\mathscr{E}$-harmonic function in $D$ satisfying \eqref{eq:2.3}.  
Assume that \eqref{e:2.3} holds.
Suppose that $|u|^p\in  D(\mathscr{E})_{e,\loc}^{\diamond}$ and it  
satisfies \eqref{eq:2.3}. 
Then $|u|^p$ is $\mathscr{E}$-subharmonic in $D$.
\item[\rm(4)]
Suppose that $u_1,u_2\in D(\mathscr{E})^{\diamond}_{e,\loc}$ satisfy 
\eqref{eq:2.3} and they are $\mathscr{E}$-subharmonic in $D$. 
Then $u_1\lor u_2\in D(\mathscr{E})^{\diamond}_{e,\loc}$ also satisfy \eqref{eq:2.3} and it is $\mathscr{E}$-subharmonic in $D$. 
The same assertion also holds for $u_1,u_2\in D(\mathscr{E})^{\diamond}_{\loc}$ satisfying \eqref{eq:2.3}. 
\end{enumerate}
\end{cor}
\begin{proof}[{\bf Proof.}]
\begin{enumerate}
\item[\rm(1)] By Theorem~\ref{thm:3.10}, for $U\Subset D$, $\{u(X_{t\land\tau_U})\}_{t\geq0}$ is a uniformly integrable $\mathbb{P}_x$-martingale for q.e.~$x\in E$. First assume that $\eta$ has a bounded derivative. 
Since $|\eta(t)-\eta(s)|\leq\sup_{\ell\in\R}|\eta'(\ell)|\cdot|t-s|$ for $t,s\in\R$, we see $\eta(u)\in D(\mathscr{E})_{e,\loc}^{\diamond}$. 
Indeed, for $V$ with $U\Subset V\Subset D$, there exists 
$u_V\in D(\mathscr{E})_e$ and $u=u_V$ ${\m}$-a.e.~on $V$. Then $\eta(u_V)-\eta(0)\in 
D(\mathscr{E})_e$. Take $\phi_V\in D(\mathscr{E})\cap C_c(E)$ such 
that $\phi_V=1$ on $V$ and set $\eta(u)_V:=\eta(u_V)-\eta(0)+\eta(0)\phi_V$. 
Then $\eta(u)_V\in D(\mathscr{E})_e$ and $\eta(u)=\eta(u)_V$ ${\m}$-a.e.~on $V$. Moreover, $\eta(u)$ satisfies \eqref{eq:2.3} in the following sense: 
\begin{align}
\1_U\mathbb{E}_{\cdot}[|\eta(\tilde{u})-\wt{\eta(u)_V}|(X_{\tau_U})]\in D(\mathscr{E}_U)_e.\label{eqq:(2.3)}
\end{align}
Indeed, 
\begin{align*}
\1_U\mathbb{E}_{\cdot}&[|\eta(\tilde{u})-\wt{\eta(u)_V}|(X_{\tau_U})]\\&\leq 
\1_U\mathbb{E}_{\cdot}[|\eta(\tilde{u})-\eta(\wt{u_V})|(X_{\tau_U})]+
\1_U\mathbb{E}_{\cdot}[|1-\phi_V|(X_{\tau_U})]\\
&\leq 
\sup_{\ell\in\R}|\eta'(\ell)|\1_U\mathbb{E}_{\cdot}[|\tilde{u}-\wt{u_V}|(X_{\tau_U})]+
\1_U\mathbb{E}_{\cdot}[|1-\phi_V|(X_{\tau_U})].
\end{align*}
The first term of the right hand side above belongs to $D(\mathscr{E}_U)_e$ 
by assumption. 
The second term of the right hand side above belongs to $D(\mathscr{E}_U)_e$ 
by \cite[Lemma~2.3]{Chen}. Since the left hand side above is excessive with respect to ${\bf X}_U$, we have \eqref{eqq:(2.3)} by \cite[Lemma~3.3]{CK}.   
Meanwhile, 
$|\eta(u)|\leq\sup_{\ell\in\R}|\eta'(\ell)|\cdot|u|+\eta(0)$ yields that 
$\{\eta(u)(X_{t\land\tau_U})\}_{t\geq0}$ is a uniformly 
integrable $\mathbb{P}_x$-submartingale for q.e.~$x\in U$ by Jensen's inequality, i.e., $\eta(u)$ is subharmonic in $D$. 
Therefore $\eta(u)$ is $\mathscr{E}$-subharmonic in $D$ satisfying 
\eqref{eq:2.3} by Theorem~\ref{thm:main}. 
Next we assume the boundedness of $u$ on $E$. Then $\eta(u)\in D(\mathscr{E})_{e,\loc}^{\diamond}$ is bounded on $E$ and it satisfies 
\eqref{eq:2.3} by Proposition~\ref{prop:Sufficient}. 
The rest of the proof is similar as above. 
\item[\rm(2)] The proof is the same as that for (1). 
\item[\rm(3)] Since $u\in D(\mathscr{E})_{e,\loc}^{\diamond}$ satisfies \eqref{eq:2.3}, $|u|\in  D(\mathscr{E})_{e,\loc}^{\diamond}$ also satisfies \eqref{eq:2.3}, because 
\begin{align*}
\1_U\mathbb{E}_{\cdot}[||\tilde{u}|-|\wt{u_V}||(X_{\tau_U})]&\leq \1_U\mathbb{E}_{\cdot}[|\tilde{u}-\wt{u_V}|(X_{\tau_U})]\in D(\mathscr{E}_U)_e
\end{align*}
implies $\1_U\mathbb{E}_{\cdot}[||\tilde{u}|-|\wt{u_V}||(X_{\tau_U})]\in D(\mathscr{E}_U)_e$ by the excessiveness of the left hand side with respect to ${\bf X}_U$ and \cite[Lemma~3.3]{CK}.  
By Theorem~\ref{thm:main}, $u$ is harmonic in $D$, hence $|u|$ is subharmonic in $D$. Applying Theorem~\ref{thm:main} again, $|u|$ is $\mathscr{E}$-subharmonic in $D$. 
Next we prove the latter assertion. 
By assumption, 
\eqref{e:2.3} holds, 
$|u|^p\in D(\mathscr{E})_{e,\loc}^{\diamond}$ holds for $p>1$
and it satisfies \eqref{eq:2.3}. 
Precisely, for $U\Subset V\Subset E$ there exists $v_{p,V}\in D(\mathscr{E})_e$ such that 
$|u|^p=v_{p,V}$ ${\m}$-a.e.~on $V$ and 
$\1_U\mathbb{E}_{\cdot}[||\tilde{u}|^p-\wt{v_{p,V}}|(X_{\tau_U})]\in D(\mathscr{E}_U)_e$. 
From this, $\mathbb{E}_x[|\tilde{u}|^p(X_{\tau_U})]<\infty$ q.e.~$x\in E$, 
since $\mathbb{E}_{\cdot}[|\wt{v_{p,V}}|(X_{\tau_U})]\in D(\mathscr{E})_e$ 
by \cite[Theorem~4.6.5]{FOT} yields $\mathbb{E}_x[|\wt{v_{p,V}}|(X_{\tau_U})]<\infty$ q.e.~$x\in E$. 
By Theorem~\ref{thm:main}, $u$ is harmonic in $D$ in the weak sense, hence 
$|u|^p$ is subharmonic in $D$ in the weak sense by Jensen's inequality. 
Applying Theorem~\ref{thm:main} again, $|u|^p$ is 
$\mathscr{E}$-subharmonic in $D$.    
\item[\rm(4)]
If both $u_1$ and $u_2$ are subharmonic, then so is $u_1\lor u_2$. 
The proof follows Lemma~\ref{lem:lattice1} and  Theorem~\ref{thm:main}.
\end{enumerate}
\end{proof}
\section{Strong Maximum Principle}\label{sec:strongmaxprin}
The following strong maximum principle slightly generalizes \cite[Theorem~2.11]{CK} without assuming the local boundedness for $\mathscr{E}$-subharmonic functions in  $D$. 
Our condition for the $\mathscr{E}$-subharmonic function $u$ in $D$  below only requires that for any open sets $U\Subset V\Subset D$
\eqref{eq:finite} and \eqref{eq:2.3} hold. 
In \cite[Theorem~2.11]{CK}, instead of \eqref{eq:finite}, 
\eqref{eq:finite*} 
is assumed. 
Note that if $u$ is locally bounded on $D$, then \eqref{eq:finite} and \eqref{eq:finite*} are equivalent.  
\begin{thm}[{Strong Maximum Principle}]\label{thm:strongMax}
Assume that ${\bf (AC)}$ holds. 
Let $D$ be an open set. Assume that the part space $(\mathscr{E}_D,D(\mathscr{E}_D))$ on $L^2(D;\m)$ is irreducible.   
Suppose that $u\in D(\mathscr{E})_{\loc}^{\diamond}$ satisfying \eqref{eq:2.3} is a finely continuous $\mathscr{E}$-subharmonic function in $D$. 
If $u$ attains a maximum at a point $x_0\in D$, then $u^+\equiv u^+(x_0)$ on $D$. If in addition $\kappa(D)=0$, then $u\equiv \tilde{u}(x_0)$ on $D$.
\end{thm}
\begin{proof}[{\bf Proof.}]
We show the proof for readers' convenience.  
Since $u^+(x_0) \ge 0$ and $\1 \in 
D(\mathscr{E})_{\loc}^{\diamond}$ is $\mathscr{E}$-superharmonic in $D$, then $u^+(x_0) - u \in 
D(\mathscr{E})_{\loc}^{\diamond}$ is a finely continuous and Borel measurable, non-negative, $\mathscr{E}$-superharmonic function in $D$. By Corollary~\ref{cor:lattice}, 
$v := u^+(x_0) - u^+ = (u^+(x_0) - u) \wedge u^+(x_0)\in D(\mathscr{E})_{\loc}^{\diamond}$ 
is a bounded non-negative $\mathscr{E}$-superharmonic function 
in $D$. We emphasize here that $u^+(x_0) - u$ is not necessarily 
locally bounded on $D$, because $u$ is not necessarily locally bounded below on $D$. So we cannot apply \cite[Corollary~2.10(iv)]{CK}.  
By \cite[Theorem~3.10]{CK} or Theorem~\ref{thm:3.10}, for any $U\Subset D$, $t\mapsto v(X_{t\land \tau_U})$ is a uniformly integrable non-negative $\mathbb{P}_x$-supermartingale for q.e.~$x\in E$. In particular, 
\begin{align*}
\mathbb{E}_x[v(X_{t\land\tau_U})]\leq v(x)\quad\text{ q.e.~}x\in D.
\end{align*}
Thanks to the non-negativity of $v$, 
\begin{align*}
\mathbb{E}_x[v(X_t):t<\tau_U]\leq v(x)\quad{\m}\text{-a.e.~}x\in U. 
\end{align*}
That is,  
\begin{align*}
p_t^Uv(x)\leq v(x)\quad {\m}\text{-a.e.}\quad x\in U.
\end{align*}
Since $v$ is bounded on $U$, we see $v\in L^2(U;{\m})$. 
Hence, for $\beta>0$  
\begin{align*}
\beta G_{\beta}^Uv(x)\leq v(x)\quad {\m}\text{-a.e.}\quad x\in U,
\end{align*}
where $G_{\beta}^U:=\int_0^{\infty}e^{-\beta t}P_t^U{\d} t$\; ($\beta>0$) is the $L^2(U;{\m})$-resolvent operators associated to $(P_t^U)_{t\geq0}$. 
From this, we can deduce 
\begin{align*}
\beta R_{\beta}^Uv(x)\leq v(x)\quad \text{ for all }\quad x\in U, 
\end{align*}
because $R_{\beta}^Uv:=\int_0^{\infty}e^{-\beta t}p_t^U\,v\,{\d} t$,  $v$ are finely continuous on $U$ and 
${\m}$ has full fine support with respect to ${\bf X}_U$ under {\bf (AC)}. 
Note here that ${\bf X}_U$ also satisfies {\bf (AC)}, because 
$\mathbb{P}_x(X_t\in A,t<\tau_U)\leq\mathbb{P}_x(X_t\in A)=0$ under ${\m}(A)=0$ for all $A\in\mathscr{B}(E)$, $x\in U$ and $t>0$. 
Taking an increasing sequence $\{U_n\}$ such that $U_n\Subset U_{n+1}\Subset D$ for all $n\in\mathbb{N}$ and $\bigcup_{n=1}^{\infty}U_n=D$. Then $\mathbb{P}_x(\lim_{n\to\infty}\tau_{U_n}=\tau_D)=1$ for all $x\in D$ (see \cite[Theorem~4.6(iv)]{Kw:maximumprinciple}). Replacing $U$ with $U_n$ and letting $n\to\infty$, we have 
\begin{align*}
\beta R_{\beta}^Dv(x)\leq v(x)\quad \text{ for all }\quad x\in D. 
\end{align*}
Since $v$ is a finely continuous Borel measurable function on $D$, 
$[0,\tau_D[\ni t\mapsto v(X_t)$ is right continuous $\mathbb{P}_x$-a.s. for all $x\in D$ (see \cite[Theorem~A.2.5]{FOT}). This yields that $v$ is excessive with respect to $\textbf{X}_D$ (see \cite[(4.11) Definition and (4.15) Exercise (v)]{Sharpe}. 
Set $Y \coloneqq \{ x \in D \;\mid\; v(x) > 0\}$ and let $(p_t^D)_{t \ge 0}$ be the transition semigroup of ${\bf X}_D$.
Since for $x \in D$, $p_t^D\1_Y(x) = \uparrow\lim_{n \to +\infty}p_t^D(nv \wedge 1)(x) \leq \1_Y(x)$, $\1_Y$ is also excessive with respect to $\textbf{X}_D$ and thus finely continuous which makes $Y$ simultaneously finely open and finely closed (see \cite{Kw:strongmax}). 
However, owing to \cite[Theorem~5.3]{Kw:maximumprinciple}, the irreducibility of $(\mathscr{E}_D, D(\mathscr{E}_D))$ implies  the connectedness of the fine topology induced on $D$ by $\textbf{X}_D$.   
Since $x_0 \in D \setminus Y$ then $Y = \emptyset$ hence $u^+ \equiv u^+(x_0)$. 

If $\kappa(D)=0$, it is enough to observe that $\1 \in 
D(\mathscr{E})_{\loc}^{\diamond}$ is $\mathscr{E}$-harmonic in $D$ hence $u(x_0)- u \in 
D(\mathscr{E})_{\loc}^{\diamond}$ is a non-negative, $\mathscr{E}$-superharmonic function.
The conclusion follows in a similar fashion as in the above. 
\end{proof}
\begin{remark}
{\rm The irreducibility assumption in Theorem~\ref{thm:strongMax} for $(\mathscr{E}_D,D(\mathscr{E}_D))$ on $L^2(D;{\m})$ holds under that ${\bf X}$ is a strong Feller process and 
$D$ is a connected open set (see \cite{Tak:exit}). 
}
\end{remark}
\section{Examples}\label{sec:Example}
\begin{example}[Brownian Motion on $\R^d$]\label{ex:BM}
\rm Let
${\bf X}^{\rm w}=(\Omega, B_t, {\P}_x)_{x\in\R^d}$ be
a $d$-dimensional Brownian motion on $\R^d$.  
The corresponding Dirichlet form $(\mathscr{E},D(\mathscr{E}))$ on $L^2(\R^d)$ is given by 
\begin{align*}
\left\{\begin{array}{rl} D(\mathscr{E})\!\!&=\;\;H^1(\R^d), \\
\mathscr{E}(f,g)\!\!&=\;\;\frac12\mathbb{D}(f,g),\quad f,g\in H^1(\R^d),\end{array}\right.
\end{align*}
where 
\begin{align*}
\left\{\begin{array}{rl}
H^1(\R^d)&:=\{f\in L^2(\R^d)\mid \frac{\partial f}{\partial x_i}\in L^2(\R^d), i\in\{1,\cdots,d\}\}, \\
\mathbb{D}(f,g)&:=
\int_{\R^d}\langle \nabla f(x),\nabla g(x)\rangle{\d} x,\quad f,g\in H^1(\R^d)
\end{array}\right.
\end{align*}
with 
$\nabla f(x):=(\frac{\partial f}{\partial x_1},\cdots, \frac{\partial f}{\partial x_d})$.  
Here the derivative $\frac{\partial f}{\partial x_i}$ is regarded as   distributional sense. We see $D(\mathscr{E})_{\loc}=D(\mathscr{E})_{\loc}^{\diamond}=D(\mathscr{E})_{\loc}^{\dag}(=H^1_{\loc}(\R^d))$. We assume $d\geq2$. 
Now we consider a special 
$\overline{\R}$-valued function 
$u$ on $\R^d$ defined by 
\begin{align*}
u(x):=|x|^{-\alpha} \quad (x\in\R^d),\qquad \alpha>0.
\end{align*}
Note that $u$ is not locally bounded around origin. 
It is easy to see that $u\in D(\mathscr{E})_{\loc}$ if and only if 
 $d>2(\alpha+1)$. Moreover, the subharmonicity 
 (resp.~superharmonicity) of 
 $u$ on $\R^d\setminus\{0\}$ in the classical sense is equivalent to $d\leq \alpha+2$ (resp.~$d\geq\alpha+2$), in 
 particular, the harmonicity of $u$ on $\R^d\setminus\{0\}$ in the classical sense is equivalent to $d=\alpha+2$. So we can deduce that 
 $u\in D(\mathscr{E})_{\loc}$ is always $\mathscr{E}$-superharmonic, since $d>2\alpha+2>\alpha+2$.
 More precisely, by integration by parts, for any non-negative 
 $\varphi\in C_c^{\infty}(\R^d)$, 
 \begin{align*}
 \mathscr{E}(u,\varphi)&=\frac12\mathbb{D}(u,\varphi)
 =\frac12\sum_{i=1}^d\int_{\R^d}\frac{\partial u}{\partial x_i}(x)
 \frac{\partial \varphi}{\partial x_i}{\d} x=\frac12\sum_{i=1}^d\int_{\R^{d-1}\setminus\{0\}}\left(\int_{\R\setminus\{0\}}\frac{\partial u}{\partial x_i}(x)
 \frac{\partial \varphi}{\partial x_i}{\d} x_i \right){\d} \overline{x}_i
\\
 &=\frac12\sum_{i=1}^d\int_{\R^{d-1}\setminus\{0\}}\left(-
 \lim_{M\to\infty,\eps\to0+}\left[\alpha|x|^{-\alpha-2}x_i
 \varphi(x) \right]_{\eps}^{M}
 -\lim_{N\to-\infty,\delta\to0-}\left[\alpha|x|^{-\alpha-2}x_i
 \varphi(x) \right]^{\delta}_{N}\right.\\
 &\hspace{5cm}\left.
 +\alpha\int_{\R\setminus\{0\}}
 (|x|^2-(\alpha+2)x_i^2)|x|^{-\alpha-4}
 \varphi(x){\d} x_i
 \right){\d} \overline{x}_i\\
 &=\frac{\alpha(d-\alpha-2)}{2}\int_{\R^d}|x|^{-\alpha-2}\varphi(x){\d} x\geq0,\quad \overline{x}_i=(x_1,\cdots, \check{x}_i,\cdots, x_d),
  \end{align*}  
  where  
 we use the continuity of 
 $x_i\mapsto x_i/|x|^{\alpha+2}$ at $0\in\R$ under 
 $d\geq2$ and $\overline{x}_i\ne0$.
By Theorem~\ref{thm:main}, we have the following: 
\begin{enumerate}
\item\label{item:superharmstable1a} $u$ is superharmonic in $\R^d$;
\item\label{item:superharmstable2a}
For  every relatively compact open subset $U$ of $\R^d$, $u(X_{\tau_U})\in L^1(\P_x)$ and
$u(x)\geq\E_x[u(X_{\tau_U})]$ for q.e.~$x\in U$.
\end{enumerate}
 \end{example}
\begin{example}[Symmetric Relativistic $\alpha$-stable Process]\label{ex:relastable}
\rm
 This example is also exposed in \cite[Example~4.1]{CK}. 
Though the statement in \cite[Example~4.1]{CK} is correct, but its 
argument is not so complete. So we expose it with a complete argument.     

Take $\alpha\in]0,2[$ and $m\geq0$.
Let ${\bf X}^{\text{\tiny\rm R},\alpha}=(\Omega,X_t,\P_x)_{x\in\R^d}$ be a
L\'evy process on $\R^d$ with
  \[ \E_{0}\bigl[e^{i \langle  \xi, X_{t}  \rangle}\bigr]=
e^{-t((|\xi|^2+m^{2/\alpha })^{\alpha/2}-m)}.\]
If $m>0$, it is called the \emph{relativistic $\alpha$-stable process with mass $m$} (see \cite{Ryz:relativisticstable}).
In particular, if $\alpha=1$ and $m>0$, it is called the \emph{relativistic free Hamiltonian process} (see \cite{HS:PP}).
When $m=0$, ${\bf X}^{\text{\tiny\rm R},\alpha}$ is nothing but the usual \emph{symmetric
$\alpha$-stable process}.
Let $(\mathscr{E}^{\text{\tiny\rm R},\alpha},
D(\mathscr{E}^{\text{\tiny\rm R},\alpha}))$
be the Dirichlet form on $L^2(\R^d)$ associated with 
${\bf X}^{\text{\tiny\rm R},\alpha}$.
Using Fourier transform $\wh{f}(x):=\frac{1}{(2\pi)^{d/2}}\int_{\R^d}e^{i\langle x,y\rangle}f(y){\d} y$, it follows
from Example~1.4.1 of \cite{FOT} that
\begin{align*}
\begin{cases}
D(\mathscr{E}^{\text{\tiny\rm R},\alpha}) :=\displaystyle{\left\{f\in L^2(\R^d)\;\Bigl|\; \int_{\R^d}
|\wh{f}(\xi)|^2\left((|\xi|^2+m^{2/\alpha})^{\alpha/2} -m \right){\d}\xi<\infty  \right\}},
 \\
\mathscr{E}^{\text{\tiny\rm R},\alpha}(f,g) :=\displaystyle{\int_{\R^d}
\wh{f}(\xi)\bar{\wh{g}}(\xi)\left((|\xi|^2+m^{2/\alpha})^{\alpha/2} -m \right){\d}\xi \quad\text{ for }f,g\in D(\mathscr{E}^{\text{\tiny\rm R},\alpha}}).
\end{cases}
\end{align*}
It is shown in \cite{CS3} that the corresponding jumping measure
satisfies
\begin{align*}
J({\d} x{\d} y)=\frac{c(x,y)}{|x-y|^{d+\alpha}}{\d} x{\d} y\ \ \text{ with }\ \ c(x,y):=\frac{A(d,-\alpha)}{2}
\Psi(m^{1/\alpha}|x-y|),
\end{align*}
where $A(d,-\alpha)=\frac{\alpha 2^{d+\alpha}\Gamma(\frac{d+\alpha}{2})}{2^{d+1}\pi^{d/2}\Gamma(1-\frac{\alpha}{2})}$, and the function $\Psi$ on $[0,\infty[$ is given by
$\Psi(r):=I(r)/I(0)$ with $I(r):=\int_0^{\infty}s^{\frac{d+\alpha}{2}-1}e^{-\frac{s}{4}-\frac{r^2}{s}}{\d} s$.
In particular,
we have
\begin{align*}
\left\{\begin{array}{cc}D(\mathscr{E}^{\text{\tiny\rm R},\alpha})&\hspace{-2cm}=\displaystyle{\left\{f\in L^2(\R^d)\;\Bigl|\;\int_{\R^d\times\R^d}|f(x)-f(y)|^2\frac{c(x,y)}{|x-y|^{d+\alpha}}{\d} x{\d} y<\infty  \right\}}, \\
\mathscr{E}^{\text{\tiny\rm R},\alpha}(f,g)&=\displaystyle{\int_{\R^d\times\R^d}
(f(x)-f(y))(g(x)-g(y))\frac{c(x,y)}{|x-y|^{d+\alpha}}{\d} x{\d} y
\quad\text{ for }f,g\in D(\mathscr{E}^{\text{\tiny\rm R},\alpha})}.
\end{array}\right.
\end{align*}
Note that
$\Psi$ is decreasing and satisfies
$\Psi(r)\asymp e^{-r}(1+r^{(d+\alpha-1)/2})$ near $r=\infty$, and $\Psi(r)=1+\Psi''(0)r^2/2+o(r^4)$ near $r=0$.  Here $f(r)\asymp g(r)$ 
near $r=\infty$ 
for positive functions $f,g$ on $[0,+\infty[$ means  
$f(r)=O(g(r))$ and $g(r)=O(f(r))$ as $r\to\infty$, more precisely 
there exist $C,r_0>0$ independent of $r>0$ satisfying $C^{-1}\leq f(r)/g(r)\leq C$ for any 
$r\in[r_0,+\infty[$ with some $r_0>0$.
More strongly, there exists $C_i=C_i(d,\alpha,m)>0$ $(i=1,2)$ such that 
\begin{align}
C_1e^{-r}(1+r^{\frac{d+\alpha-1}{2}})\leq \Psi(r)\leq C_2e^{-r}(1+r^{\frac{d+\alpha-1}{2}})\quad\text{ for all }\quad r>0.\label{eq:LowerBoundsPhi}
\end{align}
Indeed, using the change of variable $t=\frac{\sqrt{s}}{2}-\frac{r}{\sqrt{s}}$; we have
\begin{align*}
I(r)&=e^{-r}\int_{0}^{\infty}s^{\frac{d+\alpha}{2}-1}
e^{-\left(\frac{\sqrt{s}}{2}-\frac{r}{\sqrt{s}}\right)^2}{\d} s=2e^{-r}\int_{-\infty}^{\infty}\frac{(t+\sqrt{t^2+2r})^{d+\alpha}}{\sqrt{t^2+2r}}e^{-t^2}{\d} t.
\end{align*}
Then, from the elementary inequality $(a+b)^p\leq (2^{p-1}\lor1)(a^p+b^p)$ for $a,b,p>0$,  
\begin{align*}
I(r)&\leq 2e^{-r}\int_{-\infty}^{\infty}
2^{d+\alpha-1}\frac{t^{d+\alpha}+(t^2+2r)^{\frac{d+\alpha}{2}}}{\sqrt{t^2+2r}}
e^{-t^2}{\d} y\\
&\leq e^{-r}2^{d+\alpha}\int_{-\infty}^{\infty}\left(t^{d+\alpha-1}+(t^2+2r)^{\frac{d+\alpha-1}{2}} \right)e^{-t^2}{\d} t\\
&\leq2^{d+\alpha}e^{-r}\left(\int_{-\infty}^{\infty}t^{d+\alpha-1}e^{-t^2}{\d} t + (2^{\frac{d+\alpha-1}{2}-1}\lor 1)\int_{-\infty}^{\infty}
(t^{d+\alpha-1}+(2r)^{\frac{d+\alpha-1}{2}})e^{-t^2}{\d} t
\right)\\
&=2^{d+\alpha}e^{-r}(A_1+A_2r^{\frac{d+\alpha-1}{2}})\leq 2^{d+\alpha}(A_1\lor A_2)e^{-r}(1+r^{\frac{d+\alpha-1}{2}}),
\end{align*}
where $A_1:=(2^{\frac{d+\alpha-3}{2}}\lor 1+1)\int_{-\infty}^{\infty}t^{d+\alpha-1}e^{-t^2}{\d} t$ and $A_2:=(2^{\frac{d+\alpha-3}{2}}\lor 1)2^{\frac{d+\alpha-1}{2}}\sqrt{\pi}$. 
Moreover, 
\begin{align*}
I(r)&\geq 2e^{-r}\int_1^{\infty}(t^2+2r)^{\frac{d+\alpha-1}{2}}e^{-t^2}{\d} y\\
&\geq 2e^{-r}\int_1^{\infty}(1+2r)^{\frac{d+\alpha-1}{2}}e^{-t^2}{\d} y\\
&\geq 2e^{-r}(2^{\frac{d+\alpha-1}{2}-1}\land1)(1+r^{\frac{d+\alpha-1}{2}})\int_1^{\infty}e^{-t^2}{\d} t,
\end{align*}
where we use the elementary inequality $(2^{p-1}\land1)(a^p+b^p)\leq(a+b)^p$ for $a,b,p>0$. 

The content of the following proposition is stated in 
\cite[Example~4.2]{CK} 
without proof. We give here the complete proof. 
\begin{prop}\label{prop:CharacRelativiStabel}
Take $u\in L^1_{\loc}(\R^d)$. Then 
the following are equivalent to each other. 
\begin{enumerate}
\item[\rm(1)] $\Psi(m^{\frac{1}{\alpha}}|\cdot|)\left(1\land\frac{1}{|\,\cdot\,|} \right)|u|\in L^1(\R^d)$. 
\item[\rm(2)] For any $U\Subset V\Subset \R^d$, 
with $0\in V$, 
\begin{align*}
\int_{U\times V^c}|u(x)-u(y)|\frac{\Psi(m^{\frac{1}{\alpha}}|x-y|)}{|x-y|^{d+\alpha}}{\d} x{\d} y<\infty.
\end{align*}
\end{enumerate}
Moreover, take $u\in L^2_{\loc}(\R^d)$. Then the following are equivalent to each other.
\begin{enumerate}
\item[\rm(3)] $\Psi(m^{\frac{1}{\alpha}}|\cdot|)\left(1\land\frac{1}{|\,\cdot\,|} \right)|u|^2\in L^1(\R^d)$. 
\item[\rm(4)] For any $U\Subset V\Subset \R^d$, 
\begin{align*}
\int_{U\times V^c}|u(x)-u(y)|^2\frac{\Psi(m^{\frac{1}{\alpha}}|x-y|)}{|x-y|^{d+\alpha}}{\d} x{\d} y<\infty.
\end{align*}
\end{enumerate}
\end{prop}
\begin{cor}\label{cor::CharacRelativiStabel}
We have the following: 
\begin{enumerate}
\item[\rm(1)] $u\in D(\mathscr{E}^{\text{\tiny\rm R},\alpha}
)^{\diamond}_{\loc}$ implies  
$\Psi(m^{\frac{1}{\alpha}}|\cdot|)\left(1\land\frac{1}{
|\,\cdot\,|^{d+\alpha}} \right)|u|\in L^1(\R^d)$. 
\item[\rm(2)] For $u\in D(\mathscr{E}^{\text{\tiny\rm R},\alpha})_{\loc}$, 
 if  $\Psi(m^{\frac{1}{\alpha}}|\cdot-x_o|)
\left(1\land\frac{1}{|\,\cdot-x_o\,|^{d+\alpha}} \right)|u|\in L^1(\R^d)$ holds for any $x_o\in\R^d$, then $u\in D(\mathscr{E}^{\text{\tiny\rm R},\alpha})^{\diamond}_{\loc}$. 
\item[\rm(3)] $u\in D(\mathscr{E}^{\text{\tiny\rm R},\alpha})^{\dag}_{\loc}$ implies   
$\Psi(m^{\frac{1}{\alpha}}|\cdot|)\left(1\land\frac{1}{
|\,\cdot\,|^{d+\alpha}} \right)|u|^2\in L^1(\R^d)$.
\item[\rm(4)] For $u\in D(\mathscr{E}^{\text{\tiny\rm R},\alpha})_{\loc}$, 
$\Psi(m^{\frac{1}{\alpha}}|\cdot|)\left(1\land\frac{1}{
|\,\cdot\,|^{d+\alpha}} \right)|u|^2\in L^1(\R^d)$  implies $u\in D(\mathscr{E}^{\text{\tiny\rm R},\alpha})^{\dag}_{\loc}$.
\end{enumerate}
\end{cor}
\begin{proof}[{\bf Proof.}]
The proof of Corollary~\ref{cor::CharacRelativiStabel} based on 
Proposition~\ref{prop:CharacRelativiStabel} 
 is similar to 
the proof of Corollary~\ref{cor:intgrabilityequivalence} based on Lemma~\ref{lem:intgrabilityequivalence1}. So we omit the detail. 
\end{proof}
\begin{cor}\label{cor:Lipschitz}
If $u$ is a Lipschitz function on $\R^d$, then $\Psi(m^{\frac{1}{\alpha}}
|\cdot|)\left(1\land\frac{1}{|\,\cdot\,|^{d+\alpha}}\right)|u|\in L^1(\R^d)$ {\rm(}resp.~$\Psi(m^{\frac{1}{\alpha}}|\cdot|)\left(1\land\frac{1}{|\,\cdot\,|^{d+\alpha}}\right)|u|^2\in L^1(\R^d)${\rm)} under $m>0$ or 
$m=0$ with $\alpha\in]1,2[$ {\rm(}resp.~$m>0${\rm)}. 
\end{cor}
\begin{proof}[{\bf Proof.}]
For any $U\Subset V\Subset \R^d$, we see 
\begin{align*}
\int_{U\times V^c}&|u(x)-u(y)|\frac{\Psi(m^{\frac{1}{\alpha}}|x-y|)}{|x-y|^{d+\alpha}}{\d} x{\d} y\\
&\leq C_2\frac{A(d,-\alpha)}{2}\|u\|_{\rm Lip}\sigma(\mathbb{S}^{d-1})|U|\int_{{\sf d}(U,V^c)}^{\infty}
e^{-m^{\frac{1}{\alpha}}r}(1+(m^{\frac{1}{\alpha}}r)^{\frac{d+\alpha-1}{2}})r^{-\alpha}{\d} r<\infty
\end{align*}
under $m>0$ or $m=0$ with $\alpha\in]1,2[$, 
and
\begin{align*}
\int_{U\times V^c}&|u(x)-u(y)|^2\frac{\Psi(m^{\frac{1}{\alpha}}|x-y|)}{|x-y|^{d+\alpha}}{\d} x{\d} y\\
&\leq C_2\frac{A(d,-\alpha)}{2}\|u\|_{\rm Lip}^2\sigma(\mathbb{S}^{d-1})|U|\int_{{\sf d}(U,V^c)}^{\infty}
e^{-m^{\frac{1}{\alpha}}r}(1+(m^{\frac{1}{\alpha}}r)^{\frac{d+\alpha-1}{2}})r^{1-\alpha}{\d} r<\infty
\end{align*}
under $m>0$. 
Here $|U|$ is the volume of $U$. 
Thus we get the conclusion by 
Corollary~\ref{cor::CharacRelativiStabel}.
\end{proof}
\begin{remark}
{\rm We can get stronger assertions than those in Corollaries~\ref{cor::CharacRelativiStabel} and \ref{cor:Lipschitz}. 
For simplicity, we state the weaker version. 
}
\end{remark}
To prove Proposition~\ref{prop:CharacRelativiStabel}, we need the following lemma. 
\begin{lem}\label{lem:CharacRelativiStabel}
Take any $U\Subset V\Subset \R^d$. 
\begin{enumerate}
\item[\rm(1)] There exists $C_2=C_2(d,\alpha,m,U,V)>0$ such that for any $x\in U$ and $y\in V$, 
\begin{align*}
\frac{\Psi(m^{\frac{1}{\alpha}}|x-y|)}{|x-y|^{d+\alpha}}\leq C_2\Psi(m^{\frac{1}{\alpha}}|y|)\left(1\land\frac{1}{|y|^{d+\alpha}} \right).
\end{align*}
\item[\rm(2)] Suppose $0\in V$. There exists $C_1=C_1(d,\alpha,m,U,V)>0$ such that for any $x\in U$ and $y\in V$,
\begin{align*}
\frac{\Psi(m^{\frac{1}{\alpha}}|x-y|)}{|x-y|^{d+\alpha}}\geq C_1
\Psi(m^{\frac{1}{\alpha}}|y|)
\left(1\land\frac{1}{|y|^{d+\alpha}}\right).
\end{align*}
\item[\rm(3)] 
For each $x_0\in\R^d$, 
$\Psi(m^{\frac{1}{\alpha}}|\cdot-x_o|)\left(1\land\frac{1\;\,}{|\,\cdot-x_o\,|^{d+\alpha}} \right)\in L^1(\R^d)\cap L^2(\R^d)$.
\end{enumerate}
\end{lem}
\begin{proof}[{\bf Proof.}]
For the proof of (1),(2), 
applying Lemma~\ref{lem:jumpkernelequivalence}, it suffices to show that there exists $D_i=D_i(d,\alpha,m,U)>0$ with $i=1,2$ such that for $x\in U$, $y\in V^c$ 
\begin{align*}
D_1\Psi(m^{\frac{1}{\alpha}}|y|)\leq \Psi(m^{\frac{1}{\alpha}}|x-y|)\leq D_2\Psi(m^{\frac{1}{\alpha}}|y|).
\end{align*}
Since $\Psi$ is decreasing, we have from \eqref{eq:LowerBoundsPhi}
\begin{align*}
\Psi(m^{\frac{1}{\alpha}}|y|)&\geq 
\Psi(m^{\frac{1}{\alpha}}|x-y|+m^{\frac{1}{\alpha}}|x|)\\
&\geq C_1e^{-m^{\frac{1}{\alpha}}(|x-y|+|x|)}\left(1+m^{\frac{d+\alpha-1}{2\alpha}}(|x-y|+|x|)^{\frac{d+\alpha-1}{2}} \right)\\
&\geq C_1e^{-m^{\frac{1}{\alpha}}(|x-y|+|x|)}\left(1+m^{\frac{d+\alpha-1}{2\alpha}}|x-y|^{\frac{d+\alpha-1}{2}} \right)\\
&\geq C_1C_2^{-1}e^{-m^{\frac{1}{\alpha}}|x|}\Psi(m^{\frac{1}{\alpha}}|x-y|)\\
&\geq C_1C_2^{-1}e^{-m^{\frac{1}{\alpha}}\sup_{x\in U}|x|}\Psi(m^{\frac{1}{\alpha}}|x-y|)
\end{align*}
and 
\begin{align*}
\Psi(m^{\frac{1}{\alpha}}|x-y|)&\geq \Psi(m^{\frac{1}{\alpha}}|x|+m^{\frac{1}{\alpha}}|y|)\\
&\geq C_1e^{-m^{\frac{1}{\alpha}}(|x|+|y|)}\left(1+m^{\frac{d+\alpha-1}{2\alpha}}(|x|+|y|)^{\frac{d+\alpha-1}{2}} \right)\\
&\geq C_1e^{-m^{\frac{1}{\alpha}}(|x|+|y|)}\left(1+m^{\frac{d+\alpha-1}{2\alpha}}|y|^{\frac{d+\alpha-1}{2}} \right)\\
&\geq C_1C_2^{-1}e^{-m^{\frac{1}{\alpha}}|x|}\Psi(m^{\frac{1}{\alpha}}|y|)
\\
&\geq C_1C_2^{-1}e^{-m^{\frac{1}{\alpha}}\sup_{x\in U}|x|}\Psi(m^{\frac{1}{\alpha}}|y|).
\end{align*}
(3): Since $\Psi$ is continuous on $[0,1]$, 
it suffices to show $\int_1^{\infty}\Psi(m^{\frac{1}{\alpha}}r)
\left(1\land r^{-d-\alpha} \right)r^{d-1}{\d} r<\infty$ and 
$\int_1^{\infty}\Psi(m^{\frac{1}{\alpha}}r)^2
\left(1\land r^{-d-\alpha} \right)^2r^{d-1}{\d} r<\infty$. 
The calculation of the former case is 
\begin{align*}
\int_1^{\infty}&\Psi(m^{\frac{1}{\alpha}}r)
\left(1\land \frac{1}{r^{d+\alpha}} \right)r^{d-1}{\d} r\\
&\leq C_2\int_1^{\infty}e^{-m^{\frac{1}{\alpha}}r}(1+m^{\frac{d+\alpha-1}{2\alpha}}r^{\frac{d+\alpha-1}{2}})
\left(1\land \frac{1}{r^{d+\alpha}}\right)r^{d-1}{\d} r\\
&\leq C_2\int_1^{\infty}e^{-m^{\frac{1}{\alpha}}r}
(1+m^{\frac{d+\alpha-1}{2\alpha}}r^{\frac{d+\alpha-1}{2}})
r^{-\alpha-1}{\d} r
<\infty,
\end{align*}
where we use \eqref{eq:LowerBoundsPhi}. The calculation of the latter case is similar. 
\end{proof}

\begin{proof}[\bf Proof of Proposition~\ref{prop:CharacRelativiStabel}]
Thanks to Lemma~\ref{lem:CharacRelativiStabel}, the proof is similar to 
the proof of Lemma~\ref{lem:intgrabilityequivalence1}. So we omit it.  
\end{proof}

\begin{prop}\label{prop:sufficient{eq:2.3}RelativisticStable} 
Let $u$ be a Borel function satisfying $\Psi(m^{\frac{1}{\alpha}}|\cdot|)\left(1\land \frac{1}{|\,\cdot\,|^{d+\alpha}} \right)|u|^q\in L^1(\R^d)$ with $q\in]0,+\infty[$. 
Then, for any $U\Subset V\Subset\R^d$ with $0\in V$, 
\begin{align}
\sup_{x\in U}\E_x[\1_{V^c}|u|^q(X_{\tau_U})]<\infty.\label{eq:1Rel}
\end{align}
From this, for a nearly Borel function $u\in D(\mathscr{E}^{\text{\tiny\rm R},\alpha})_{\loc}$ 
satisfying $\Psi(m^{\frac{1}{\alpha}}|\cdot|)\left(1\land \frac{1}{|\,\cdot\,|^{d+\alpha}} \right)|u|\in 
L^1(\R^d)$ {\rm(}resp.~$\Psi(m^{\frac{1}{\alpha}}|\cdot|)\left(1\land \frac{1}{|\,\cdot\,|^{d+\alpha}} \right)|u|^2\in L^1(\R^d)${\rm)}, we have \eqref{eq:Finiteness} 
for any $U\Subset V\Subset \R^d$ and $u_V\in D(\mathscr{E}^{\text{\tiny\rm R},\alpha})$ satisfying $u=u_V$ a.e.~on $V$ with respect to ${\bf X}^{{\rm\tiny R},\alpha}$, 
in particular, \eqref{eq:2.3} for $u\in D(\mathscr{E}^{\text{\tiny\rm R},\alpha})^{\diamond}_{\loc}$ holds, 
provided $\sup_{x\in U}\E_x[\tau_U]<\infty$ for any $U\Subset \R^d$.
\end{prop}
\begin{proof}[{\bf Proof.}]
The proof is quite similar to the proof of Proposition~\ref{prop:sufficient{eq:2.3}} by  
\cite[Lemma~4.1(2)]{KimKuwae:GeneralAnal}. 
We omit it. 
\end{proof}
By Propositions~\ref{prop:CharacRelativiStabel}, \ref{prop:sufficient{eq:2.3}RelativisticStable} and  Theorem~\ref{thm:main}, 
for an open set $D$ and a nearly Borel $\overline{\R}$-valued 
q.e.~finely continuous 
function $u$ on $\R^d$  
satisfying $\Psi(m^{1/\alpha}|\cdot|)(1\wedge |\cdot|^{-d-\alpha})|u|
\in L^1(\R^d)$ and $u\in L^1_{\loc}(\R^d)$, the following are equivalent.
\begin{enumerate}
\item\label{item:subharmstable1b} $u$ is subharmonic in $D$;
\item\label{item:subharmstable2b}
For  every relatively compact open subset $U$ of $D$, $u(X_{\tau_U})\in L^1(\P_x)$ and
$u(x)\leq\E_x[u(X_{\tau_U})]$ for q.e.~$x\in U$;
\item\label{item:subharmstable3b}
$u\in D(\mathscr{E}^{\text{\tiny\rm R},\alpha})^{\diamond}_{\loc}$ and
$$
\int_{\R^d\times\R^d}
(u(x)-u(y))(v(x)-v(y))\frac{\Psi(m^{1/\alpha}|x-y|)}{|x-y|^{d+\alpha}}{\d} x{\d} y\leq0
$$
for every non-negative $v\in C^\infty_c(D)$.
\end{enumerate}
Moreover, if $u$ is (finely) continuous, the above equivalence can  be formulated without exceptional sets.

\end{example}

\noindent
{\bf Acknowledgments} 
The authors would like to thank Professor Kaneharu Tsuchiya for carefully reading the draft of this paper. We would also like to thank the anonymous referees, whose comments have greatly improved the quality of this paper. 

\bigskip
\noindent
{\bf Conflict of interest.} 
The authors have no conflict of interest relative to the content of this article.

\bigskip
\noindent
{\bf Data Availability Statement.} Data sharing is not applicable to this article as no datasets were generated or analyzed during the current study.

\providecommand{\bysame}{\leavevmode\hbox to3em{\hrulefill}\thinspace}


\end{document}